\documentclass[12pt,a4paper]{amsart}
\usepackage{amsfonts}
\usepackage{bbold}
\usepackage{amsthm}
\usepackage{amsmath}
\usepackage{amscd}
\usepackage[latin2]{inputenc}
\usepackage{t1enc}
\usepackage[mathscr]{eucal}
\usepackage{indentfirst}
\usepackage{graphicx}
\usepackage{graphics}
\usepackage{pict2e}
\usepackage{epic}
\numberwithin{equation}{section}
\usepackage[margin=2.9cm]{geometry}

\theoremstyle{plain}
\newtheorem{Th}{Theorem}[section]
\newtheorem{Lemma}[Th]{Lemma}
\newtheorem{Pro}[Th]{Proposition}

\newtheorem{Prop}[Th]{Proposition}
\newtheorem{Cor}[Th]{Corollary}

 \theoremstyle{definition}
\newtheorem{Def}[Th]{Definition}

\newtheorem{Rem}[Th]{Remark}
\newtheorem{Ex}[Th]{Example}
\newtheorem{?}[Th]{Problem}
\newcommand{\C}{\mathbb{C}}
\newcommand{\Z}{\mathbb{Z}}\newcommand{\R}{\mathbb{R}}
\newcommand{\la}{\lambda}
\newcommand{\aut}{\operatorname{aut}}

\DeclareMathOperator{\diam}{diam}
\DeclareMathOperator{\vv}{v}
\DeclareMathOperator{\cc}{c}
\DeclareMathOperator{\e}{e}\DeclareMathOperator{\Pm}{pm}\DeclareMathOperator{\ch}{ch}
\DeclareMathOperator{\inj}{inj}

\begin{document}
\author[P.E. Frenkel]{P\'eter E.\ Frenkel}
\address{Alfr\'ed R\'enyi Institute of Mathematics, Hungarian Academy of Sciences \\
1053 Budapest, Hungary \\
Re\'altanoda u. 13-15. \& ELTE E\"{o}tv\"{o}s Lor\'{a}nd University \\ Faculty of Science \\  Institute of Mathematics\\ 1117 Budapest, Hungary
\\ P\'{a}zm\'{a}ny P\'{e}ter s\'{e}t\'{a}ny 1/C }

\email{frenkelp@cs.elte.hu}
\thanks{This project has received funding from  the European Research Council
(ERC) under the European Union's Horizon 2020 research and innovation
program (grant agreement No.\ 648017), from the MTA R\'enyi Lend\"ulet
Groups and Graphs research group, and from the Hungarian National
Research, Development and Innovation Office -- NKFIH, OTKA grants no.\
K104206 and K109684.}

\title{Convergence of graphs with intermediate density}

\begin{abstract} We propose a  notion of graph convergence that interpolates between the Benjamini--Schramm convergence of bounded degree graphs and the dense graph convergence developed  by L\'aszl\'o Lov\'asz 
and his coauthors.
We prove that spectra of graphs, and also some important graph parameters such as numbers of colorings or matchings, behave well in convergent graph sequences. Special attention is given to graph sequences of large essential girth, for which asymptotics of coloring  numbers are explicitly calculated. We also treat numbers of matchings in approximately regular graphs.

We introduce tentative limit objects that we call graphonings because they are common generalizations of graphons and graphings. Special forms of these, called Hausdorff and Euclidean graphonings, involve geometric measure theory. We construct Euclidean graphonings that provide  limits of hypercubes and of finite projective planes, and, more generally, of a wide class of  regular sequences of large essential girth.  For any convergent sequence of large essential girth, we construct   weaker limit objects: an involution invariant probability measure on the sub-Markov space of consistent measure sequences (this is unique), or an  acyclic reversible sub-Markov kernel on a probability space (non-unique). We also pose some open problems.
\end{abstract}

\maketitle
\tableofcontents
\bf Notations and terminology. \rm Graphs are finite, simple, and undirected, unless otherwise specified. 
 On $k$ nodes, the complete graph, cycle, path, and path with a fork at one end   is denoted by $K_k$, $C_k$, $P_k$, and $D_k$, respectively.
For a graph $G=(V(G), E(G))$, we write $\vv (G)=|V(G)|$ and $\e(G)=|E(G)|$. A graph $F$ has $\cc(F)$ connected components, out of which $\cc_{\ge 2}(F)$ have at least two nodes. The neighborhood (i.e., set of neighbors) of  a node $o$ is written $N(o)$.

The number of homomorphisms and injective homomorphisms from $F$ to $G$ is denoted by $\hom(F,G)$ and $\inj(F,G)$, respectively. The number of automorphisms of $F$ is $\aut F$. The symbols $\times$ and $\square$ stand for the categorical (or weak) direct product and the Cartesian sum of graphs, respectively.

The product of $\sigma$-algebras is denoted by $\otimes$. We write a.e.\ for ``almost every(where)'' and a.s.\ for ``almost surely'', i.e., ``with probability 1''. The indicator of an event $A$ is $\mathbb 1_A$.

\section{Homomorphism densities and graph convergence}\label{Def}
The two most developed graph limit theories are the Benjamini--Schramm limit theory of bounded degree graphs and the dense graph limit theory
 developed by Borgs, Chayes,  Lov\'asz, T.\  S\'os, Szegedy, and Vesztergombi. The convergence of dense graphs is defined in terms of 
  homomorphism densities. The convergence of
bounded degree graphs is defined in terms of neighborhood statistics, but this easily translates into convergence of homomorphism frequencies.  We now propose a common generalization that works for both cases and also for intermediate density.
\begin{Def}\label{t} An \it admissible pair \rm is a pair $(G,d)$, where $d\ge 1$ and $G$ is a graph with  all degrees $\le d$. For a connected graph $F$ and an admissible pair $(G,d)$, we define the \it homomorphism density \rm
\[t(F,G,d)=\frac{\hom (F,G)}{\vv (G)d^{\vv (F)-1}}\in [0,1].\] We extend this to arbitrary $F$ by making it multiplicative:
\[t(F,G,d)=\frac{\hom (F,G)}{\vv (G)^{\cc(F)}d^{(\vv -\cc)(F)}}\in [0,1].\]
An \it admissible sequence \rm 
is a  sequence of admissible pairs. 
An admissible sequence $(G_n,d_n)$ is \it convergent \rm if the number sequence $t(F, G_n, d_n)$ converges for any (or, equivalently, any connected) graph $F$.
\end{Def}

\begin{Rem}
Note that \[t(F,G,\vv (G))=\frac {\hom (F,G)}{\vv (G)^{\vv (F)}}=t(F,G)\] is the usual homomorphism density. Thus, a sequence of the form $(G_n,\vv (G_n))$ --- which is always admissible --- is convergent precisely if
$(G_n)$ is a  convergent dense graph sequence.

Note also that for connected $F$ we have \[t(F,G,d)=t^*(F,G)/d^{\vv (F)-1},\] where $t^*(F,G)=\hom (F,G)/\vv (G)$ is the usual homomorphism frequency, so if $d_n=d$ does not depend on $n$, then an admissible  sequence $(G_n,d)$ is convergent precisely if $(G_n)$ is a Benjamini--Schramm convergent graph sequence (alternatively called a locally convergent graph sequence). 

When $F$ is a forest, the normalization used in Definiton~\ref{t} is similar to the one used by Bollob\'as and Riordan~\cite{BR} and by Borgs, Chayes, Cohn, and Zhao~\cite{BCCZ1, BCCZ2}. However, for general $F$, our normalization is quite different. The goal in those papers was to generalize  dense graph convergence to the sparse case, but no attempt was made to also include Benjamini--Schramm convergence in a  unified treatment. In the present approach, both extremes are included as special cases. This is also reflected in the limit objects --- generalized graphons ---, which are $L_p$ graphons in \cite{BCCZ1, BCCZ2} but graphonings in Section~\ref{limobj} of the present paper. Admittedly, the results presented in this paper are less conclusive.
\end{Rem}

\begin{Rem}\label{subgraph} Let $(G,d)$ be an admissible pair. 
Removing an edge from a connected  graph $F$ without destroying connectivity cannot decrease $t(F,G,d)$. Removing a  vertex of degree 1 from a connected graph $F$ cannot either. Thus, we have  $t(F,G,d)\le t(F', G,d)$ if $F'\subseteq F$ are connected graphs. 
\end{Rem}

\begin{Pro}\label{tree} 
  Let $(G,d)$ be an admissible pair. Let $F$ be a graph. 
  Then  we have $t(F,G,d)=1$ if and only if at least one of the following holds.

  \begin{itemize}\item[(a)]
 $F$ is an edgeless graph, or

 \item[(b)]  $F$ is a forest and $G$ is $d$-regular, or

 \item[(c)]  $F$ is bipartite and $G$ is a disjoint union of complete bipartite graphs $K_{d,d}$.
 \end{itemize}
\end{Pro}

\begin{proof}  We may assume that $F$ and $G$ are connected.

If any of (a), (b), (c) holds, then an easy induction on $\vv(F)$ shows that $\hom(F,G)=\vv(G) d^{\vv(F)-1}$ and the claim follows.

For the converse, assume that
$t(F,G,d)=1$.

  If (a) does not hold, then $F$ contains $K_2$ as a subgraph, thus $2\e (G)/(\vv(G)d)=t(K_2,G,d)\ge  t(F,G,d)=1$ and therefore $G$ is $d$-regular.

 If $F$ contains an odd cycle $C_{2k+1}$, then consider the path  $P_{2k+1}=C_{2k+1}-e$ for an edge $e\in E(C_{2k+1})$. We have $$t(C_{2k+1}, G, d)=1=t(P_{2k+1}, G,d),$$
thus  $$\hom(C_{2k+1}, G)=\hom(P_{2k+1}, G).$$
 But there exists a  homomorphism $\phi:P_{2k+1}\to G$ such that images of the two endnodes coincide. Such a  $\phi$ does not extend to $C_{2k+1}$ because $G$ has no loops. This contradiction proves that $F$ is bipartite.

 If $F$ contains an even cycle $C_{2k}$ for some $k\ge 2$, then a similar argument shows that in $G$, the two endnodes
 of any walk of length $2k-1$ are joined by an edge. It follows that this holds for 3 in place of $2k-1$, and thus for  any odd length as well.
 But $G$ has no loops, so it must be bipartite. It is connected, so it is a complete bipartite graph. It is $d$-regular, so $G\simeq K_{d,d}$.
\end{proof}

\begin{Ex}\label{dirprod}
Let $(\Gamma_i, \delta_i)$ be admissible pairs ($i=1,2,\dots$). Set $G_n=\Gamma_1\times \cdots\times\Gamma_n$ and $d_n=\delta_1\cdots\delta_n$. Then the sequence $(G_n, d_n)$ is convergent. The homomorphism density $t(F,G_n, d_n)$ converges to $\prod_{i=1}^\infty t(F,\Gamma_i,\delta_i)$.
\end{Ex}

\begin{proof}
We have $\hom(F,G_n)=\prod_{i=1}^n\hom(F,\Gamma_i)$ and  $\vv(G_{ n})=\prod_{i=1}^n\vv(\Gamma_i)$, whence $$t(F,G_{n}, d_n)=\prod_{i=1}^nt(F,\Gamma_i,\delta_i).$$ This is decreasing and therefore convergent as $n\to\infty$.
\end{proof}

\begin{Cor}
Let $(\Gamma, \delta)$ be an admissible pair. Then the sequence $(\Gamma^{\times n}, \delta^n)$ is convergent. The homomorphism density $t(F,\Gamma^{\times n}, \delta^n)$ converges to $1$ if $t(F,\Gamma,\delta)=1$ and to zero otherwise.
\end{Cor}

\begin{Ex}\label{smallcomp}
 If $G$ is a disjoint union of graphs  $G^i$  ($i=1,\dots, \vv(G)/d$) of size $d$, then $$t(F,G,d)=\frac d{\vv(G)}\sum_{i=1}^{\vv(G)/d}t(F,G^i)$$ for all connected $F$.  We can think of each $G^i$ as  a point in the compact graphon space $\widetilde{\mathcal W
    }_0$ of L.\ Lov\'asz and B.\ Szegedy~\cite{L, LSz}, and consider the uniform probability measure on these $\vv(G)/d$ points. We can think of $\widetilde{\mathcal W
    }_0$ as sitting in $[0,1]^\infty$, each graphon $W$ being represented by its profile of homomorphism densities $t(F,W)$ with connected $F$. A sequence $(G_n, d_n)$,  such that $G_n$ is a disjoint union of graphs of size $d_n$, is convergent if and only if the barycenters of the corresponding probability measures form a convergent sequence.  This is strictly weaker than the weak convergence of the probability measures themselves. If $(G_n,d_n)$ converges, then the limit can be represented by the limiting barycenter (which is unique), or any subsequential weak limit measure (which is non-unique in general, but each one has the correct barycenter).
\end{Ex}
Further examples of convergent sequences are regular sequences of large essential girth, such as hypercube graphs, large grid graphs, incidence graphs of finite projective spaces, and suitable random nearly regular graphs.  See Subsections~\ref{REGULAR} and \ref{GIRTH}.

\subsection{Injective homomorphism densities} It is sometimes useful to count injective, rather than arbitrary,  homomorphisms. We introduce injective homomorphism densities. Even in the dense case, our normalization deviates slightly from the standard one in Lov\'asz's monograph~\cite{L}.

\begin{Def} Let 
 $(G,d)$ be an admissible  pair. 
 For a connected graph $F$, we define the \it injective homomorphism density \rm
\[t_{\inj}(F,G,d)= \frac{\inj (F,G)}{\vv (G)d(d-1)^{\vv (F)-2}}\in [0,1]\] unless $F$ is a  single point, in which case $t_{\inj}(F,G, d)=1$. We extend this to arbitrary $F$ by making the denominator multiplicative:
\[t_{\inj}(F,G,d)= \frac{\inj (F,G)}{\vv (G)^ {\cc(F)}d^{\cc_{\ge 2}(F)}(d-1)^{(\vv-\cc-\cc_{\ge 2}) (F)}}\le \prod_i t_{\inj}(F_i,G,d),\] where the $F_i$ are the connected components of $F$.
\end{Def}

\begin{Rem}\label{subgraphinj}  Let $d> 1$ and let $(G,d)$ be an admissible  pair. Removing an edge from a connected  graph $F$ without destroying connectivity cannot decrease $t_{\inj}(F,G,d)$. Removing a  vertex of degree 1 from a connected graph $F$ cannot either. Thus, $t_{\inj}(F,G,d)\le t_{\inj}(F', G,d)$ if $F'\subseteq F$ are connected graphs.
\end{Rem}

\begin{Pro}\label{homversusinj} For any fixed connected graph  $F$, we have \[t(F,G,d)-t_{\inj}(F,G,d)=O(1/d)
,\] where the constant in the $O$ depends only on $F$.
\end{Pro}

\begin{proof} For $F=K_1$, both densities are 1 and the claim is trivial. For $\vv (F)\ge 2$, we have $$t(F,G,d)\ge  \frac{\inj(F,G)}{\vv (G)d^{\vv (F)-1}}=t_{\inj}(F,G,d)\left(1-\frac1d\right)^{\vv (F)-2},$$
whence
\[t_{\inj}(F,G,d)-t(F,G,d)\le t_{\inj}(F,G,d)\left(1-\left(1-\frac1d\right)^{\vv (F)-2}\right)\le\frac{\vv (F)-2}d.
\]
 On the other hand, we have the well-known formula \[\hom (F,G)=\sum _{F'}\inj (F',G)
 ,\] where $F'$ runs over the quotients of $F$.  Note that quotients of connected graphs are connected, and proper quotients have fewer vertices than the original graph. Thus, \[ t(F,G,d)= \sum_{F'} \frac{\inj(F',G)}{\vv (G)d^{\vv (F)-1}}\le\sum_{F'} \frac{t_{\inj}(F',G,d)}{d^{\vv (F)-\vv (F')}}=t_{\inj}(F,G,d)+O\left(\frac1d\right).\]
\end{proof}

\begin{Cor}\label{hominj} An admissible sequence $(G_n,d_n)$ with $d_n\to\infty$ is convergent precisely if the injective homomorphism density $t_{\inj} (F,G_n,d_n)$ converges for all connected graphs $F$. If this is the case, then \[\lim_{n\to\infty} t_{\inj} (F,G_n,d_n)=\lim_{n\to\infty}t(F,G_n,d_n)\] for any  connected $F$.
\end{Cor}

This is well known in the dense case: homomorphism and injective homomorphism densities are almost the same.

It will be useful to also compare injective and componentwise injective homomorphisms.

\begin{Pro}\label{compinj} Let  $F$ have connected components $F_i$. Then  we have \begin{align*}0\le\left({\prod_it_{\inj}(F_i,G,d)}\right)-t_{\inj}(F,G,d)&\le\\ \le \frac1{\vv(G)}\sum_{F'}t_{\inj}(F',G,d)+
O\left(\frac1{\vv(G)^2}\right)&=O\left(\frac1{\vv(G)}\right)
,\end{align*} where $F'$ runs over the quotients of $F$ such that each $F_i$ maps injectively to $F'$ and $\cc(F')=\cc(F)-1$. The constant in the $O$ depends only on $F$.
\end{Pro}

\begin{proof} We have  \[\prod_i\inj (F_i,G)=\sum _{F'}\inj (F',G)
 ,\] where $F'$ runs over the quotients of $F$ that arise by only identifying nodes from distinct components.  We always have $(\vv-\cc)(F')\le (\vv-\cc)(F)$ and $\cc_{\ge 2}(F')\le \cc_{\ge 2}(F)$, hence  \[\prod_i t_{\inj}(F_i,G,d)= \sum_{F'} \frac{\inj(F',G)}{\vv (G)^{\cc(F)}d^{\cc_{\ge 2} (F)}(d-1)^{(\vv-\cc-c_{\ge 2})(F)}}\le\sum_{F'} \frac{t_{\inj}(F',G,d)}{\vv(G)^{\cc (F)-\cc (F')}}
 \] and the claim follows.
\end{proof}

\subsection{Rooted homomorphism densities}
\begin{Def} Let $(F,o)$ and $(G,p)$ be rooted graphs, where $F$ is connected. Let $\hom((F,o), (G,p))$ be the number of homomorphisms of $F$ into $G$ that map $o$ to $p$. 
If $(G, d)$ is admissible, we define the \it rooted homomorphism density \rm \[t((F,o), (G,p),d)=\frac{\hom((F,o), (G,p))}{d^{\vv (F)-1}}\in [0,1].\]
\end{Def}

\begin{Rem}\label{rootedunrooted} For any connected rooted graph $(F,o)$ and any admissible pair $(G, d)$, we have \[t(F,G,d)=\mathbb E
t((F,o), (G,p),d),\] where $p$ is a uniform random node of $G$.
\end{Rem}

\subsection{Regular  sequences}\label{REGULAR}
\begin{Def} Let $0\le\alpha\le 1$. The admissible sequence $(G_n,d_n)$ is \it $\alpha$-regular \rm if the degree of a uniform random vertex of $G_n$, divided by $d_n$, tends stochastically to $\alpha$. 
\end{Def}

If the graph $G_n$ is $\alpha d_n$-regular for every $n$, then of course the sequence $(G_n,d_n)$ is $\alpha$-regular. Let us look at less trivial examples.

\begin{Ex}
Let $G_n$ be the $d_n$-dimensional grid graph with $n\times\cdots\times n$ points.
Then $\vv(G_n)=n^{d_n}$ and $$\e(G_n)=d_nn^{d_n-1}(n-1),$$
so $$t(K_2,G_n,2d_n)=\e(G_n)/(\vv(G_n)d_n)=(n-1)/n\to 1,$$ i.e., the sequence $(G_n,2d_n)$ is $1$-regular, cf.\ Proposition~\ref{reg} below. If the sequence $d_n$ either stabilizes to some $d$ or tends to $\infty$, then $(G_n,2d_n)$ is convergent, cf.\ Subsection~\ref{GIRTH}.
\end{Ex}

The case when $d_n\to\infty$ can be  generalized as follows.

\begin{Ex}\label{generalCartesian} Consider a triangular array  $(\Gamma_{ni},\delta_{ni})$ $(1\le i\le n)$ of admissible pairs with normalized  average degree $\alpha_{ni} =t(K_2,\Gamma_{ni}, \delta_{ni})$. Set $G_n=\Gamma_{n1}\square\cdots\square\Gamma_{nn}$ and $d_n=\delta_{n1}+\dots+\delta_{nn}$. Assume that \begin{equation}\label{kicsiny} \max_{1\le i\le n}\delta_{ni}/d_n\to 0\end{equation} and the weighted average $$\frac1{d_n}\sum_{i=1}^n\delta_{ni}\alpha_{ni}\to\alpha.$$ Then the sequence $(G_{ n}, d_n)$ is $\alpha$-regular.\end{Ex}

\begin{proof} Let $X_{ni}$ be the degree of a uniform random node in $\Gamma_{ni}$, divided by $\delta_{ni}$. Then $X_{ni}$ is a random variable with range in $[0,1]$, and $\mathbb EX_{ni}=\alpha_{ni}$. The degree of a  uniform random node in $G_{ n}$, divided by $d_n$, is $$X_n=(\delta_{n1}X_{n1}+\dots+\delta_{nn}X_{nn})/d_n,$$ where the $X_{ni}$ are independent. We have $\mathbb EX_n\to\alpha$ and $$\mathbb D^2X_n=(\delta_{n1}^2\mathbb D^2X_{n1}+\dots+\delta_{nn}^2\mathbb D^2X_{nn})/d_n^2\le \frac1{d_n^2}\sum_{i=1}^n\delta_{ni}^2\to 0.$$ Thus $X_n\to\alpha$ stochastically as claimed,  by Chebyshev's inequality.
\end{proof}

Again we refer to Subsection~\ref{GIRTH} where it will be proved that such a sequence $(G_n, d_n)$ of Cartesian sums is always convergent.

\begin{Cor}\label{Cartesian} Let $(\Gamma, \delta)$ be admissible with normalized  average degree $t(K_2,\Gamma,\delta)=\alpha$. Then the sequence $\left(\Gamma^{\square n}, n\delta\right)$ is $\alpha$-regular.\end{Cor}

Regular sequences can be characterized in terms of homomorphism densities.

\begin{Pro}\label{alphareg}
For an admissible sequence $(G_n,d_n)$, the following are equivalent.
\begin{itemize}
\item[(a)] The sequence  $(G_n,d_n)$ is $\alpha$-regular.


\item[(b)] We have  $t(K_2, G_n,d_n)\to \alpha$ and  $t(P_3, G_n,d_n)\to \alpha^{2}$ as $n\to\infty$.

\item[(c)] For all  forests $F$, we have  $t(F, G_n,d_n)\to \alpha^{\e(F)}$ as $n\to\infty$.

\item[(d)] For all rooted trees $(F, o)$, we have  $t((F,o), (G_n, p_n),d_n)\to \alpha^{\e(F)}$ stochastically, as $n\to\infty$. Here $p_n$ is a  uniform random node of $G_n$.
\end{itemize}
\end{Pro}
Note that the statement (d) for $F=K_2$ is exactly the same as (a).
\begin{proof}
Let $X_n$ be the degree of a uniform random node of $G_n$, divided by $d_n$. Then $X_n$ is a random variable with values in $[0,1]$.
The equivalence of (a) and (b) is clear since 
we have  $t(K_2,G_n,d_n)=\mathbb EX_n$  and  $t(P_3,G_n,d_n)=\mathbb EX_n^{2}$.


As (d) $\Rightarrow$ (c)  $\Rightarrow$  (b) is trivial, it suffices  to show that (a)  implies (d). We use induction on $\vv(F)$. The case $\vv(F)=1$ is trivial.
 Let $\vv(F)\ge 2$. Let $o_i$ $(i=1,\dots ,k)$ be the neighbors of $o$ in $F$, i.e., $N(o)=\{o_1,\dots, o_k\}$.
Let $F_i$ be the connected component of $F-o$ containing $o_i$.

Let $\epsilon>0$. By the induction hypothesis, for $n\ge n_0(\epsilon)$ there exists
an $S_n\subset V(G_n)$, with $|S_n|<\epsilon^2 \vv(G_n)$, such that for all $q\in V(G_n)-S_n$ and all $i$, we have  $$\left|\sqrt[\e(F_i)]{t((F_i,o_i), (G_n, q),d_n)}- \alpha
\right|<\epsilon.$$ Let $T_n$ be the set of nodes in $G_n$ that have at least $\epsilon d_n$ neighbors in $S_n$. Since all nodes in $S_n$ have at most $d_n$ neighbors, we have $|T_n|\le |S_n|/\epsilon<\epsilon \vv(G_n).$ For all $p\in V(G_n)-T_n$, we have
$$0\le \hom((F,o),(G_n,p))-\hom((F,o,N(o)), (G_n, p, V(G_n)-S_n))\le k\epsilon d_n^{\vv(F)-1}.$$ Let $U_n$ be the set of nodes in $G_n$ whose degree divided by $d_n$ is not in $(\alpha-\epsilon, \alpha+\epsilon)$. For $n\ge n_0(\epsilon)$, we have $|U_n|<\epsilon \vv(G_n)$ by (a). For all $p\in V(G_n)-T_n-U_n$, we have $$
((\alpha-2\epsilon)d_n)^{\e(F)}\le \hom((F,o,N(o)), (G_n, p, V(G_n)-S_n))\le ((\alpha+\epsilon)d_n)^{\e(F)}$$
and therefore
$$(\alpha-2\epsilon)^{\e(F)}\le t((F,o),(G_n,p), d_n)\le (\alpha+\epsilon)^{\e(F)}+k\epsilon.$$
This is true for all $\epsilon>0$, $n\ge n_0(\epsilon)$, and $p\in V(G_n)-T_n-U_n$, where $|T_n|+|U_n|<2\epsilon \vv(G_n)$. Statement (d) follows.
\end{proof}

\begin{Pro}\label{reg}
For an admissible sequence $(G_n,d_n)$, the following are equivalent.
\begin{itemize}
\item[(a)] The sequence  $(G_n,d_n)$ is $1$-regular.

\item[(b)] The average  degree in $G_n$ is asymptotically $d_n$.

\item[(c)] We have  $t(K_2, G_n,d_n)\to 1$ as $n\to\infty$.

\item[(d)] For all forests $F$, we have  $t(F, G_n,d_n)\to 1$ as $n\to\infty$.
\end{itemize}
\end{Pro}

\begin{proof}
The equivalence of (a) and (b) is clear since $(G_n,d_n)$ is admissible.

Observe that $t(K_2,G,d)$ is the average degree in $G$, divided by $d$. This shows the equivalence of (b) and (c).

Since trivially (d) $\Rightarrow$ (c), it suffices to prove  (a) $\Rightarrow$ (d). This follows from Proposition~\ref{alphareg} and Remark~\ref{rootedunrooted}.
\end{proof}

\subsection{Sequences with large essential girth} \label{GIRTH}
\begin{Def} The graph sequence $(G_n)$ \it has large girth \rm if, for any $k\ge 3$, we have $\inj(C_k,G_n)=0$ for $n\ge n_0(k)$.
The admissible sequence $(G_n, d_n)$  \it has  large essential girth \rm if, for all $k\ge 3$, we have $t_{\inj}(C_k, G_n,d_n)\to 0$ as $n\to\infty$.
\end{Def}

\begin{Rem}
If $(G_n,d_n)$ has large essential girth and $G_n'$ is a spanning subgraph of $G_n$, then
$(G'_n,d_n)$ has large essential girth.
\end{Rem}

\begin{Rem}\label{walk}
The injective homomorphism density of  a cycle in a  graph $G$ satisfies
\begin{equation}\label{NBW}t_{\inj}(C_k, G,d)\le\frac {N(G)}{\vv(G)d(d-1)^{k-2}}, \end{equation} where $N(G)$ is the number of non-backtracking walks of  length $k-1$ in $G$ whose starting point and endpoint are adjacent. The fraction on the right hand side  has the following interpretation in terms of random walks. Choose $v_0\in V(G)$ uniformly at random. With probability $\deg (v_0)/d$, let $v_1$ be a  neighbor of $v_0$ chosen uniformly at random. With probability  $1-\deg (v_0)/d$, do not define  $v_1$. If $i\ge 2$ and $v_{i-1}$ is defined, then with probability
 $(\deg (v_{i-1})-1)/(d-1)$, let $v_i$ be a  neighbor of $v_{i-1}$, distinct from $v_{i-2}$,  chosen uniformly at random. With probability  $1-(\deg (v_{i-1})-1)/(d-1)$, do not define  $v_i$.  Then
the right hand side of \eqref{NBW} is the probability that $v_{k-1}$ is defined and adjacent to $v_0$.

When $G$ is $d$-regular, $v_i$ is almost surely defined
for every $i$.
\end{Rem}

For example, we look at two classical examples of regular  graphs with intermediate density: hypercubes and projective planes (and their generalizations below).  Let $Q_n=\{0,1\}^n$ be the hypercube graph.

\begin{Pro}\label{cubeproj}
\begin{itemize}
\item[(a)]  Let $q_n\to\infty$ and let $G_n$ be the (bipartite) incidence graph of points and hyperplanes in a projective space of order $q_n$ and dimension $r_n$. Let $$d_n=\left(q_n^{r_n}-1\right)/(q_n-1).$$  Then $G_n$ is $d_n$-regular and $(G_n, d_n)$ has large essential girth.

\item[(b)]    Consider a triangular array  $(\Gamma_{ni},\delta_{ni})$ $(1\le i\le n)$ of admissible pairs. Let $G_n$ be a subgraph of $\Gamma_{n1}\square\cdots\square\Gamma_{nn}$, and let $d_n=\delta_{n1}+\dots+\delta_{nn}$. As in Example~\ref{generalCartesian}, assume that \eqref{kicsiny} holds.
Then the  sequence $(G_n, d_n)$ 
has large essential girth.

\end{itemize}
\end{Pro}
In particular, if $(G,d)$ is admissible, then the sequence $\left(G^{\square n},nd\right)$ has large essential girth. For example,
 the  sequence $(Q_n,n)$ has large essential girth. More generally, if $G_n$ is any finite subgraph of the $n$-dimensional grid $\Z^n$, then the sequence $(G_n,2n)$ has large essential girth.

\begin{proof} (b) We make use of Remark~\ref{walk}. It suffices to prove that if we do in $G_n$ the random walk defined there, then for any fixed $k\ge 3$, the  probability that $v_{k-1}$ exists and is adjacent to $v_0$ tends to 0 as $n\to\infty$. 
Clearly, the Hamming distance of $v_0$ and $v_{k-1}$ will be $k-1\ge 2$ with probability tending to 1 as $n\to\infty$.
Indeed, when doing (at most) $k-1$ steps, the probability that there will be two steps in the same coordinate goes to zero as $n\to\infty$, because of \eqref{kicsiny}.

\medskip

(a) We omit the subscript $n$ for easier reading.  The regularity claim is clear since any hyperplane has $d$ points and any point is on $d$ hyperplanes.

We have $$\inj(C_k,G)\le \left(\frac{q^{r+1}-1}{q-1}\right)^{k/2} \left(\frac{q^{r-1}-1}{q-1}\right)^{k/2},$$ whence $$t_{\inj}(C_k, G, d)\le \frac{(q^{r+1}-1)^{k/2-1}(q^{r-1}-1)^{k/2}}{2(q^r-1)^{k-1}}\le\frac{q^{r-1}-1}{2(q^r-1)}<\frac1{2q}\to 0$$ as $n\to\infty$.
\end{proof}

Statement (a) is maybe a  bit surprising since in a  large girth 1-regular sequence $(G_n, d_n)$, the number  $\vv(G_n)$ of nodes would have to be superpolynomial in $d_n$, whereas for projective spaces we have  $d=(q^r-1)/(q-1)$ and $$\vv(G)=2(q^{r+1}-1)/(q-1)=2(qd+1)<2\left(d^{r/(r-1)}+1\right).$$  This means in particular that we cannot delete $o(\vv(G_n)d_n)$ edges from $G_n$ to make the sequence have large girth (if $r_n\ge 2$ for all $n$). This is in contrast to the bounded degree case. Thus, the word `essential' is essential. Another instance of this will be Proposition~\ref{funny}.


In other words, $\vv(G_n)/d_n$ can go to $\infty$ arbitrarily slowly (compared to $\vv(G_n)$ and $d_n$) in a  1-regular sequence of large essential girth: the dimension $r$ and thus the cardinality of the projective space can grow arbitrarily fast compared to the order $q$, while we have $\vv(G)/d\sim 2q$.
A  sequence of large essential girth can thus be almost dense. It it easy to see, however, that it cannot be dense:

\begin{Pro}\label{notdense}
If $(G_n,d_n)$ is an admissible sequence with  large essential girth, where  $\vv(G_n)\to\infty$ and $d_n=O(\vv(G_n))$, then
$$\frac {d_n}{\vv(G_n)}t(K_2,G_n,d_n)=\frac{2\e(G_n)}{\vv^2(G_n)}\to 0$$ as $n\to\infty$.
\end{Pro}

\begin{proof} We have $$\frac{2\e(G_n)}{\vv^2(G_n)}=t(K_2,G_n)\le \sqrt[4]{t(C_4,G_n)}=O\left( \sqrt[4]{t(C_4,G_n,d_n)}\right),$$ where $$t(C_4,G_n,d_n)=t_{\inj}(C_4,G_n,d_n)+O(1/d_n)$$ and $t_{\inj}(C_4,G_n,d_n)\to 0$ as $n\to\infty$. Also, $$\frac {d_n}{\vv(G_n)}t(K_2,G_n,d_n)\le \frac {d_n}{\vv(G_n)}.$$ The Proposition follows.
\end{proof}

Further important examples of regular sequences of large essential girth are obtained by considering random graphs. Let $G(n,d)$ be the random  (almost) $d$-regular multigraph generated by the configuration model: we take $n$ nodes with $d$ legs emanating from each node, and take a uniform random perfect matching on the $dn$ legs (if $dn$ is odd, leave out a  leg). Let $G(n,d)^{\mathrm{simp}}$ be the underlying simple graph.

\begin{Pro}\label{random} Fix $\epsilon>0$.
Let $d_n=O(n^{1-\epsilon})$ and let the random graph $G_n$ have the distribution of
$G(n,d_n)^{\mathrm{simp}}$.
Then, for any joint distribution of the $G_n$,  the sequence $(G_n,d_n)$  a.s.\ is 1-regular and has large essential girth.
\end{Pro}

\begin{proof} Let $L$ be the proportion of loops among the edges of $G(n,d_n)$. It is easy to see that $\mathbb EL^2=O(1/n^2)$, whence $L\to 0$ a.s. 

Let $r$ be so large that $\sum (d_n/n)^r<\infty$.

For easier reading, we omit the subscript $n$ from $d_n$ and $G_n$.

Let $M$ be the proportion of edges in $G(n,d)$ that have  a parallel edge.
We have $\mathbb EM^r=O((d/n)^r)$, whence $M\to 0$ a.s. 
Thus, the sequence is a.s.\ 1-regular.

Let $k\ge 3$.
Using Proposition~\ref{compinj}, we have \[\mathbb E t_{\inj}^r(C_k, G, d)\le\mathbb E t_{\inj}(C_k^r, G, d)+
\frac1{n}\sum_{F'}\mathbb Et_{\inj}(F',G,d)+
O\left(\frac1{n^2}\right),\]  where  $C_k^r$ is the  union of $r$ pairwise disjoint $k$-cycles, and $F'$ runs over quotients of $C_k^r$ into which each component $C_k$ maps injectively, such that $F'$ has $r-1$ components. It is easy to see that $\mathbb E t_{\inj}(C_k^r, G, d)=O((d/n)^r)$ and $\mathbb Et_{\inj}(F',G,d)=O((d/n)^{r-1})$ for each $F'$, whence $\sum\mathbb E t_{\inj}^r(C_k, G, d)<\infty$ and therefore $ t_{\inj}(C_k, G, d)\to 0$ a.s.
\end{proof}

This concludes our set of examples of sequences with large essential girth.
Putting  together Propositions~\ref{subgraphinj}, \ref{homversusinj},  and  \ref{alphareg},   we obtain
\begin{Prop}\label{girthdensity}
If $d_n\to\infty$ and $(G_n,d_n)$ is $\alpha$-regular and has large essential girth, then
 \begin{itemize}
  \item[(a)] the homomorphism density $t(F,G_n,d_n)$ and the injective  homomorphism density $t_{\inj}(F,G_n,d_n)$ converge to $\alpha^{\e(F)}$ for any tree $F$ and to 0 for any other connected $F$;

\item[(b)] the sequence $(G_n,d_n)$ is convergent.
\end{itemize}
 \end{Prop}

In particular, hypercubes, or --- more generally --- Cartesian powers of a fixed graph, or  grids of size $n\times\cdots\times n$ where both  $n$ and the dimension tend to $\infty$, or point-hyperplane incidence graphs of projective spaces whose order tends to $\infty$, or the random graphs of Proposition~\ref{random}, form convergent sequences.

\section{Convergence of spectra}
Let $\sigma_{G,d}$ be the uniform probability measure  on the $\vv (G)$ numbers $\la/d$, where $\la$ runs over the eigenvalues of $G$. If  $(G, d)$ is admissible, then all eigenvalues of $G$ are in $[-d,d]$, so  $\sigma_{G,d}$ is supported on $[-1,1]$.
We have $$\epsilon^2\sigma_{G,d}(\{x:|x|\ge \epsilon\})\le\int_{-1}^1 x^2\mathrm d \sigma_{G,d}(x)=\frac{2\e(G)}{\vv (G)d^2}=\frac{t(K_2,G,d)}d\le \frac1d.$$
This proves
\begin{Pro}\label{spectrum} Let $(G_n, d_n)$ be an admissible  sequence with $d_n\to\infty$. Then the measure $\sigma_n=\sigma_{G_n,d_n}$ converges weakly to the Dirac measure at 0. More precisely, $$\sigma_n((-\epsilon, \epsilon))\ge 1-1/(\epsilon^2d_n)$$ for all $n$ and all $\epsilon>0$.
\end{Pro}
This is probably well known but I couldn't find a reference.

Up to now, we used only the second moment of $\sigma_{G,d}$, but to get more precise results for convergent sequences, we shall need the other moments as well. The zeroth moment is 1, the first moment is 0, and we have $$\int x^k\mathrm d\sigma_{G,d}(x)=\frac{\hom (C_k,G)}{\vv (G)d^k}=\frac{t(C_k,G,d)}d$$ for $k\ge 3$. In fact, this formula holds for $k\ge 1$ if we agree that $C_2=K_2$ and $C_1$ is a  node with a loop. We infer
\begin{Lemma}\label{W}
Let $(G_n,d_n)$ be convergent and $g:[-1,1]\to \R$ be continuous. 
 Then $$d_n\int_{-1}^1 x^2g(x)\mathrm d \sigma_{G_n,d_n}(x)$$
converges as $n\to\infty$.
\end{Lemma}

\begin{proof}
For $g(x)=x^k$, $k\ge 0$, the statement is clear from the preceding discussion. The general case follows by the Weierstrass approximation theorem.
\end{proof}

For Benjamini--Schramm convergent graph sequences ($d_n$ independent of $n$), it is well known that the spectral measure converges weakly, to a nontrivial measure in general. 

 For convergent dense graph sequences,  C.\ Borgs, J.\ T.\ Chayes, L.\ Lov\'asz, V.\ T.\ S\'os,  and K.\ Vesztergombi (see \cite[Subsection 6.3]{BCLSV} and \cite[Section 11.6]{L}) have given a much more precise description of the limiting behavior of the spectrum  than the one in Proposition~\ref{spectrum}. Namely, the $k$-th largest (resp.\ $k$-th smallest) eigenvalue,  divided by the number of nodes, converges to the $k$-th largest (smallest) eigenvalue of the limiting graphon, which is nonnegative (nonpositive).

We shall now show that these two results  (bounded degree and  dense) carry over to intermediate density  --- at least partially: we don't (yet) have limit objects, cf.\ Section~\ref{limobj}.

Let $G$ have all degrees $\le d$. Let $1\le r\le \vv (G) $ be an integer. Let $\sigma_{G,d,r}$ and $\sigma'_{G,d,r}$ be the uniform probability measures on the numbers $\la/d$, where $\la$ runs over the $r$ largest and $r$ smallest eigenvalues of $G$, respectively. These measures are supported on $[-1,1]$. Note that $\sigma_{G,d,\vv (G)}=
\sigma'_{G,d,\vv (G)}=\sigma_{G,d}$. For any $r$, the probability measures $\sigma_{G,d,r}$ and $\sigma'_{G,d,r}$ are  the restrictions of the measure $(\vv(G)/r)\sigma_{G,d}$ to the intervals $\left[\la_r/d,1\right]$ and $\left[-1,\la_{\vv(G)-r+1}\right]$, respectively.

\begin{Th}\label{specconv}
Let $(G_n,d_n)$ be a convergent sequence with $d_n\to\infty$.  Let $1\le r_n\le \vv (G_n)$ $(n=1,2,\dots)$ be  integers such that $r_nd_n/\vv (G_n)$ converges to a positive limit $\alpha$. Then the measures $\sigma_n=\sigma_{G_n,d_n,r_n}$ and $\sigma'_n=\sigma'_{G_n,d_n,r_n}$ converge weakly to  probability measures $\sigma$ supported  on $[0,1]$ and $\sigma'$ supported  on $[-1,0]$, respectively.
\end{Th}

\begin{proof}
We only treat  $\sigma_n$ since everything  works the same way for $\sigma_n'$.

Let $\la_r$ be the $r$-th largest eigenvalue of the graph $G$ with all degrees $\le d$.  We have $$0=\sum_{i=1}^{\vv(G)}\la_i\le (\vv(G)-r)\la_r+rd,$$ whence \begin{equation*}\frac{\la_r}d\ge-\frac{r}{\vv(G)-r}.\end{equation*}

Thus, 
  the measure $\sigma_n$ is supported on the halfline with left endpoint
 $$-\frac{r_n}{\vv(G_n)-r_n}\to 0$$ since
 $r_n/\vv(G_n)\sim\alpha/d_n\to 0$.
 Thus, it suffices to show that for any $0 <a<b<1$, we have  \begin{equation}\label{nemlotyog}\liminf\sigma_n([a,1])\ge \limsup  \sigma_n([b,1]).\end{equation} Let $g:[-1,1]\to [0,1]$ be continuous, nondecreasing, $g(a)=0$, $g(b)=1$.

 For all $n$, either $\sigma_n([a,1])=1$ or $$\sigma_n([a,1])= (\vv(G_n)/r_n)\sigma_{G_n,d_n}([a,1]).$$ Hence,  $\liminf\sigma_n([a,1])=1$ or  $$\liminf\sigma_n([a,1])\ge \liminf\left( \frac{d_n}\alpha\int g\mathrm d \sigma_{G_n,d_n}\right)=\frac1\alpha\lim \left(d_n\int g\mathrm d \sigma_{G_n,d_n}\right).$$ On the other hand, for all $n$, we have  $\sigma_n([b,1])\le 1$
 and $$\sigma_n([b,1])\le\int g\mathrm d\sigma_n\le \frac{\vv(G_n)}{r_n}\int g\mathrm d\sigma_{G_n,d_n}.$$
Hence,  $\limsup\sigma_n([b,1])\le 1$ and  $$\limsup\sigma_n([b,1])\le \limsup\left( \frac{d_n}\alpha\int g\mathrm d \sigma_{G_n,d_n}\right)=\frac1\alpha\lim \left(d_n\int g\mathrm d \sigma_{G_n,d_n}\right),$$ and the inequality~\eqref{nemlotyog} follows.\end{proof}

\section{Graph polynomials}\label{graphpol}
The convergence of a  sequence $(G_n,d_n)$ was defined in Section~\ref{Def} by the convergence of certain graph parameters, the homomorphism densities. This forces certain further parameters to converge (sometimes only under further conditions); such parameters are called \it estimable \rm (some parameters are only estimable for a certain class of convergent sequences). Theorem~\ref{specconv} can be thought of as an estimability statement. In this section, we present some more estimable parameters.

Following the paper \cite{csf} by P.\ Csikv\'ari and the present author, let $f$ be an isomorphism-invariant monic multiplicative graph polynomial of linearly bounded exponential type. I.e., \begin{itemize}\item for every graph $G$, a monic polynomial $f(G,x)\in\C [x]$ of degree $\vv (G)$ is given,
\item $f(G_1,x)=f(G_2,x)$ if $G_1\simeq G_2$,

\item $f(G_1\cup G_2, x)=f(G_1,x)f(G_2,x)$ for any disjoint union,

 \item \[f(G, x+y)=\sum_{S\subseteq V(G)} f(G[S],x)f(G[V(G)-S],y)\] for all $G$, and finally

 \item  \begin{equation}\label{bound} \sum \{|f'(G[S],0)|  :  v\in S\subseteq V(G), |S|=t\}\le (cd)^{t-1}\end{equation} for all $G$ with maximal degree $\le d$, all $v\in V(G)$, and all $t\ge 1$, with a constant $c$ depending only on $f$.\end{itemize}

Examples include
the chromatic,  adjoint, and Laplacian characteristic polynomials, and also the modified matching polynomial  defined as
$$M(G,x)=\sum_{k=0}^{\lfloor \vv (G)/2\rfloor} (-1)^km_k(G)x^{\vv(G)-k},$$
where $m_k(G)$ is the number of matchings in $G$ that consist of  $k$ edges.

 The characteristic polynomial $f(G,x)=\det(xI-A_G)$  of the adjacency matrix of $G$ is not a valid  example  because it  is not of exponential type. Nevertheless everything that follows, including Theorem~\ref{polconv} below, applies to this case in a trivial way; in fact, much more is true, even without assuming graph convergence, as we have seen in Proposition~\ref{spectrum}. 

We wish to study the distribution of roots of $f(G,x)$. By \cite[Theorem 1.6]{csf}, we can choose a constant $C$ depending only on $f$ such that for any $G$, all roots have absolute value $\le Cd$. It is shown there that $C=7.04\cdot c$ is an appropriate choice if $c$ is the constant in \eqref{bound}. For some of the specific graph polynomials  mentioned above, smaller appropriate values of $C$ are known. 

Let $p_k(G)$ be the $k$-th power sum of the roots of $f(G,x)$.
By \cite[Theorem 5.6.(b)]{csf}, for each $k\ge 1$, there exist constants $c_k(F)$ such that \begin{equation}\label{powersum}p_k(G)=\sum_{2\le \vv (F)\le k+1}c_k(F)\inj(F,G)\end{equation} for all $G$, where $F$ runs over the isomorphism classes of connected graphs
.
We also have $$p_0(G)=\vv (G)=c_0(K_1)\inj(K_1, G),$$ where $c_0(K_1)=1$.

Let $\nu_{G,d}$ be the uniform probability measure on the points $\lambda/d$, where $\la$ runs over the roots of $f(G,x)$. This measure is supported
on the disc of radius $C$ and has $k$-th holomorphic moment
\begin{equation}\label{moment}
\begin{aligned}
\int z^k d\nu_{G,d}(z)=\frac 1{\vv (G)}\sum_{f(G,\la)=0} \frac  {\la ^k}{d^k}=\frac {p_k(G)}{\vv (G)d^k}=\sum_{2\le \vv (F)\le k+1}c_k(F)\frac{\inj (F,G)}{\vv (G)d^k}=\\
=\sum_{2\le \vv (F)\le k+1}c_k(F)\frac{(d-1)^{\vv (F)-2}}{d^{k-1}}t_{\inj}(F,G,d)
\end{aligned}
\end{equation}for $k\ge 1$.

\begin{Th}\label{polconv} Let  $d_n\to\infty$. Let $(G_n, d_n)$ be a  convergent sequence, or, more generally, an admissible  sequence such that $t(F, G_n, d_n)$ converges whenever $c_{\vv (F)-1}(F)\ne 0$. Write $$t(F)=\lim_{n\to\infty}
 t (F,G_n,d_n)$$ for the limiting homomorphism density. Set $\nu_n=\nu_{G_n,d_n}$.
\begin{enumerate}

\item For all $k\ge 0$, we have 
 $$\int z^k d\nu_n(z)\to\sum_{ \vv (F)= k+1}c_k(F) t (F)
$$ as $n\to\infty$.

\item For any function $g(z)$ that is continuous for $|z|\le C$ and harmonic for $|z|<C$, the integral $\int g(z)d\nu_n(z)$ converges as $n\to\infty$.

    \item For any $|\xi|>C$, the normalized absolute value  $$\frac{\sqrt[\vv (G_n)]{|f(G_n,\xi d_n)|}}{d_n}$$ of $f$         converges to a positive limit.

        \item\label{weak} If $f(G_n,x)$ has only real roots for all $n$, then $\nu_n$ converges weakly.
\end{enumerate}
\end{Th}

For the bounded degree case, the analogous theorem is \cite[Theorem 1.10]{csf}, which, in turn, was a  generalization (with a simpler proof) of  the result of M.\ Ab\'ert and T.\ Hubai \cite[Theorems 1.1, 1.2]{ah}, who first discovered this phenomenon in the case of the chromatic polynomial. For the dense case, essentially the same was proved by P.\ Csikv\'ari, J.\ Hladk\'y, T.\ Hubai and the author in \cite[Theorems 1.4, 1.5, 4.3]{csfhh}, using the approach of \cite{csf}.  The proof carries over to  intermediate density almost unchanged.

\begin{proof}
\begin{enumerate}
\item\label{momentconv} The 0-th moment is always 1. Let  $k\ge 1$. From \eqref{moment}, we have \begin{equation*}\begin{aligned}\int z^k d\nu_{n}(z)=\sum_{2\le \vv (F)\le k+1}c_k(F)\frac{(d_n-1)^{\vv (F)-2}}{d_n^{k-1}}t_{\inj}(F,G_n,d_n)\to  \sum_{ \vv (F)= k+1}c_k(F) t (F)
\end{aligned}\end{equation*} as $n\to\infty$.

\item\label{harm} The claim follows from \eqref{momentconv} because $g(z)$ can be uniformly approximated by real parts of polynomials.

\item  For any $G$, $d$, and $\xi$, we have $$\log \frac{\sqrt[\vv (G)]{|f(G,\xi d)|}}{d}=\frac1{\vv (G)}\log\frac{|f(G,\xi d)|}{d^{\vv (G)}}=\frac1{\vv (G)}\log\prod_{i=1}^{\vv (G)} \left|\xi-\frac{\la_i}d\right|,$$ where the $\la_i$ are the roots of $f(G,x)$.  The last expression  can be rewritten as $$\frac1{\vv (G)}\sum_{i=1}^{\vv (G)} \log\left|\xi-\frac{\la_i}d\right|=\int g(z)\mathrm d\nu_{G,d},$$ where $g(z)=\log |\xi - z|$. The claim now follows from the previous  statement \eqref{harm}.

\item The claim follows from \eqref{momentconv} because each $\nu_n$ is supported on the interval $[-C,C]$.
\end{enumerate}
\end{proof}

\subsection{Number of proper colorings (large essential girth case)}
We now wish to prove, for  intermediate density  graph sequences of large essential girth, a qualitative variant of Ab\'ert and Hubai's \cite[Theorem 1.4]{ah} about the asymptotic number of proper colorings. They only treated the large girth case, but gave an explicit bound on the error in their formula.

Let $\ch (G,x)$ be the chromatic polynomial of the graph $G$. I.e., for integral $q\ge 0$, $\ch(G,q)$ is the number of proper $q$-colorings of $G$.
\begin{Th} 
 Let $(G_n,d_n)$ be a sequence of large essential girth, such that $d_n\to\infty$ and 
  $$ t(K_2,G_n,d_n)=\frac{2\e(G_n)}{\vv(G_n)d_n}\to 
 t(K_2)$$ as $n\to\infty$. Let $|\xi|\ge 8$.
Then \begin{equation}\label{chvalue}\frac{\sqrt[\vv (G_n)]{|\ch(G_n,\xi d_n)|}}{|\xi|d_n}\to \exp(-t(K_2)\Re (1/2\xi)).
\end{equation}
\end{Th}

\begin{proof}
We have \begin{equation}
\label{hubai}\begin{aligned}
\log\frac{\sqrt[\vv (G)]{|\ch(G,\xi d)|}}{|\xi|d}=\int_{|z|\le C}\log \left|1-\frac z\xi\right|d\nu_{G,d}(z)=\\=-\sum_{k=1}^\infty \frac{1}k\Re \int_{|z|\le C}\left(\frac z\xi\right)^kd\nu_{G,d}(z),\end{aligned}\end{equation} where $\nu_{G,d}$ is the uniform probability measure on the $\vv (G)$ points $\la/d$ for which $\ch (G,\la)=0$, and $C<8$ is Sokal's constant such that $|\la|\le Cd$ for all $\la$. The series on the right hand side of \eqref{hubai} converges uniformly in $G$ and
 $d$.

 By  \cite[Theorem 6.6]{csf}, in the formula \eqref{powersum} for the power sum $p_k(G)$ of the roots of $\ch(G,x)$, the coefficient $c_k(F)$ is 0 unless $F$ is 2-connected. On the other hand, $t_{\inj}(F,G_n,d_n)\to 0 $ if $F$ contains a cycle.  Thus,  $c_k(F) t_{\inj}(F,G_n,d_n)\to 0 $  unless $F$ is a 2-connected tree, i.e., $F=K_2$. Note also that $c_1(K_2)=1/2$ because $p_1(G)=\e(G)=\inj(K_2,G)/2$.
 From formula \eqref{moment}, we see that $\int z^k\mathrm d\nu_n(z)$ tends to $t(K_2)/2$ for $k=1$ and to 0 for $k\ge 2$.

 Putting all this together, the logarithm of the  left hand side of \eqref{chvalue} tends to $-\Re (t(K_2)/2\xi)$, as claimed.
\end{proof}

\subsection{Matching measure and graph convergence} \label{matching_measure}
In this subsection, we prove intermediate degree analogs of some results of the recent paper \cite{acfk} by Ab\'ert, Csikv\'ari, Kun and the author. Contrary to the bounded degree case treated there, large girth will not play any role in relation to matchings.

\begin{Def} Let $G$ be a graph  and let $m_k(G)$ denote the number of matchings
  of size $k$. Then the \it matching polynomial \rm  $\mu(G,x)$ is defined as follows:
$$\mu(G,x)=\sum_{k=0}^{\lfloor \vv(G)/2\rfloor} (-1)^km_k(G)x^{\vv (G)-2k}.$$
Note that $m_0(G)=1$.
Let $d>0$ be an upper bound on all degrees in $G$.
The \it matching measure \rm $\rho_{G,d}$  is defined to be the uniform probability distribution on the points $\la/\sqrt d$, where $\la$ runs over the roots of $\mu(G,x)$ (with multiplicity).
\end{Def}

The fundamental theorem for the matching polynomial is the following.

\begin{Th}[Heilmann and Lieb \cite{hei}]
\begin{itemize}
\item[(a)]
The roots of the matching polynomial
$\mu(G,x)$ are real.

\item[(b)] If  $d\ge 2 $ is an upper bound for all degrees in $G$,
then  all roots of $\mu(G,x)$ have absolute value $\le 2\sqrt{d-1}$.
\end{itemize}
\end{Th}

Many graph parameters related to matchings can be read off from 
 the matching measure, for example,  the number $$\mathbb M(G)=\sum_{k=0}^{\lfloor \vv(G) /2\rfloor}m_k(G)$$ of all matchings and the number $\Pm (G)=m_{\vv(G)/2}$ of perfect matchings. The latter is zero if $\vv(G)$ is odd.

\begin{Prop}\label{matchpar}\begin{equation}\label{M} 
\log\frac{ {\mathbb M(G)}^{2/\vv(G)}}{d}
=\int_{-2}^2\log \left(\frac1d+x^2\right)\mathrm d\rho_{G,d}(x).\end{equation}

\begin{equation}\label{pm}
\log\frac{ {\mathbb \Pm(G)}^{2/\vv(G)}}{d}
=
2\int_{-2}^2\log|x|\mathrm d\rho_{G,d}(x)
.\end{equation}
\end{Prop}

\begin{proof}
\eqref{M}
 The number of matchings in $G$ is \[\mathbb M(G)=\sum_{k=0}^{\lfloor \vv(G)/2\rfloor}m_k(G)=|\mu(G,\sqrt{-1})|
.\] Thus,  \begin{align*}\frac{\log \mathbb M(G)}{\vv(G)}-\frac12\log d=\frac{\log |\mu(G,\sqrt{-1})|}{\vv (G)}-\frac12\log d=\\ =\int_{-2}^2\log \left|\frac{\sqrt{-1}}{\sqrt d}-x\right|\mathrm d\rho_{G,d}(x)
=\frac12\int_{-2}^2\log \left(\frac1d+x^2\right)\mathrm d\rho_{G,d}(x)
.\end{align*}

\medskip

\eqref{pm}
The number of perfect matchings in $G$ is \[\Pm (G)=|\mu(G,0)|.\]
Thus,  \[\frac{\log \Pm (G)}{\vv(G)}-\frac12\log d=\frac{\log |\mu(G,0)|}{\vv(G)}-\frac12\log d=\int_{-2}^2\log|x|\mathrm d\rho_{G,d}(x)
.\]
\end{proof}
Let $$w(x)=\frac{ \sqrt{4-x^2}}{2\pi}\qquad (-2\le x\le 2)
   $$
 denote Wigner's semicircle density function. The  \it semicircle distribution on the interval $[-2\beta, 2\beta]$ \rm is the distribution of $\beta X$, where $X$ is a random variable with density $w$.


\begin{Th} \label{wc} Let $d_n\to\infty$.
Let $(G_n, d_n)$ be  an admissible sequence with matching measures  $\rho_n=\rho_{G_n,d_n}$.
\begin{itemize}
\item[(a)] If  $t(F,G_n,d_n)$ is convergent   for any tree $F$, then the sequence of matching measures  $\rho_n$
converges weakly to a probability measure $\rho$ on $[-2,2]$. Moreover, we have
\[\limsup_{n\to \infty}\log\frac{ {\mathbb M(G_n)}^{2/\vv(G_n)}}{d_n}
\le 2\int_{-2}^2 \log |x|\mathrm d\rho(x).\]
\item[(b)] If the sequence $(G_n,d_n)$ is $\alpha$-regular, then $\rho$ is the semicircle distribution on the interval $\left[-2\sqrt \alpha , 2\sqrt \alpha \right]$, and we have \begin{equation}\label{limsup}\limsup_{n\to \infty}
\frac{
\mathbb{M}(G_n)^{2/\vv(G_n)}}{d_n}\le \frac\alpha e.\end{equation}

\end{itemize}
\end{Th}

For example, the matching polynomial of the complete graph  $K_n$ is the $n$-th Hermite polynomial, so (b) recovers the ancient fact that root distributions of Hermite polynomials converge to the semicircle law \cite{F, G,  KM, Sz}. Similarly, complete bipartite graphs $K_{n,n}$  yield Laguerre polynomials.

When each graph $G_n$ is $d_n$-regular, the first statement in (b) has been also independently obtained by Ab\'ert, Csikv\'ari and Hubai with a different proof (unpublished), and the inequality in (b) follows from the much stronger result of Davies, Jenssen, Perkins and Roberts~\cite[Theorem 4]{DJPR}.

When each graph $G_n$ is $d_n$-regular and bipartite,
$\Pm(G_n)^{2/\vv(G_n)}\sim d_n/e$, which is
well known to follow from classical results of Br\`egman ($\le $) and Schrijver ($\ge$), see \cite[pp. 311--312]{LP}.
From the inequality in (b), we see that $\mathbb{M}(G_n)^{2/\vv(G_n)}\sim d_n/e$ as well, and we only need Schrijver's lower bound
 \[\Pm(G)^{2/{\vv (G)}}\ge \frac{(d-1)^{d-1}}{d^{d-2}}\sim \frac de\qquad (d\to\infty)\] on the number of perfect matchings to get this.
Note that Propp's 1999 survey on the enumeration of matchings cites \cite{CGP} for  the asymptotic formula   for the number of perfect matchings of the hypercube, and asks for a formula for the number of all matchings~\cite[Problem 19]{P}.

Leaving  regular graphs,
note that statement (a)  applies in particular to the special case when $(G_n,d_n)$ is convergent. 
The first claim  in (a), for the special case of convergent dense graph sequences, is \cite[Theorem 4.3]{csfhh} of Csikv\'ari, Hladk\'y, Hubai and the author.

We prove Theorem~\ref{wc}.
\begin{proof}
 By the Heilmann--Lieb Theorem, the measures $\rho_n$ are all supported on $[-2,2]$. We shall exploit the relation between the modified and the ordinary matching polynomial:   $M(G,x^2)=x^{\vv(G)}\mu(G,x)$.
Let $\nu_{G,d}$ be the uniform probability measure on the points $\la/d$, where $\la$ runs over the roots of the modified matching polynomial $M(G,x)$. This measure is supported on the interval $[0,4]$.

 There is a very nice interpretation of the $2k$-th power sum
 of the roots of the matching polynomial $\mu(G,x)$. It counts the number of
closed tree-like walks of length $2k$ in the graph $G$ \cite[Chapter 6]{god3}.
Note that for $k\ge 1$, this 
 is twice the $k$-th power sum of the roots of the modified matching polynomial $M(G,x)$. Thus, in the formula \eqref{powersum} written for the graph polynomial $M(G,x)$, the coefficient  $c_k(F)$ is half the number of tree-like walks of length $2k$ in $F$ that use all edges of $F$, divided by $\aut F$.  Thus, $c_{\vv (F)-1}=0$ unless $F$ is a  tree.

\medskip

(a)
Let $\nu_n=\nu_{G_n,d_n}$.
By Theorem~\ref{polconv}\eqref{weak}, $\nu_n$ converges weakly as $n\to \infty$.  But
from  $\nu_{G,d}$ we get $\rho_{G,d}$ by decreasing the mass at 0 by 1/2 and then relocating  the mass of any point $x$ to both points $\pm\sqrt x$, so as to get a  probability measure again. This operation clearly preserves weak convergence.
Thus, $\rho_n$ also converges weakly
 to a measure $\rho$.

 Let   $u(x)=2\log |x|$ and  $$u_k(x)=\log\left(\frac 1k +x^2\right)$$
for $k=1,2,\dots$.
Then \[\log\frac{ {\mathbb M(G)}^{2/\vv(G)}}{d}
\le \int_{-2}^2 u_k \mathrm d\rho_{G,d} 
\] if $d\ge k$.
 Thus, for any $k$,
\[\limsup_{n\to\infty}\log\frac{ {\mathbb M(G_n)}^{2/\vv(G_n)}}{d_n}
\le\lim_{n\to\infty}\int_{-2}^2 u_k \mathrm d\rho_{n}= \int_{-2}^2 u_k\mathrm  d\rho 
,\]
 since the measures $\rho_{n}$ are  supported on the  compact interval $[-2,2]$  not depending on $n$, and  $u_k$ is  continuous and  bounded on $[-2,2]$.

Since $u_k\ge u_{k+1}$ and $u_k\to  u$ pointwise, the claim follows using the Monotone Convergence Theorem.

\medskip

(b) Matching measures  are symmetric about 0, and so is the semicircle measure, so it suffices to show convergence of even moments of $\rho_n$ to those of the semicircle law. Let $k\ge 1$. By Theorem~\ref{polconv}, we have
\begin{equation}\label{catalan}\int_{-2}^2 x^{2k}\mathrm d\rho_n(x)=
2\int_0^4 x^{k}\mathrm d\nu_n(x)\to\sum_{\vv (F)=k+1} 2c_k(F)t(F),
\end{equation} where $$t(F)=\lim_{n\to \infty} t(F,G_n,d_n)=\alpha^k$$ by Proposition~\ref{alphareg}.
So the limit in \eqref{catalan}  is $\alpha^k$ times  the number of nonisomorphic pairs $(F, \gamma)$, where $F$ is a tree with $k+1$ nodes and $\gamma$ is an Eulerian trail in the graph $\tilde F$ which is $F$ with   all edges doubled. These pairs $(F,\gamma)$ correspond to Dyck words of length $2k$, so their number is the Catalan number $$\frac1{k+1}\binom{2k}k=\int_{-2}^2 x^{2k}w(x)dx=\mathbb EX^{2k},$$  where $X$ has density $w$.
Therefore $$\int_{-2}^2x^{2k}\mathrm d\rho(x)=\lim_{n\to\infty}\int_{-2}^2 x^{2k}\mathrm d\rho_n(x)=\alpha^k\mathbb EX^{2k}=\mathbb E\left(\sqrt \alpha X\right)^{2k}
$$ as claimed.


The inequality~\eqref{limsup} is immediate from statement (a)  and the fact that $$\int_{-2}^{2}w(x)\log |x|\mathrm dx=-\frac12,$$ cf.\ \cite[integral 4.241.9]{GR}.
\end{proof}





\subsection{Spectral measure rescaled (regular, large girth case)}
We know from Proposition~\ref{spectrum} that scaling down the spectrum by the degree bound $d\to\infty$ leads to trivial behavior in terms of weak convergence. What happens if we only scale down by $\sqrt {d}$ ?

Let $\Sigma_{G,d}$ be the uniform probability measure on the $\vv(G)$ points $\la/\sqrt d$, where $\la$ runs over the eigenvalues of $G$. If all degrees in $G$ are $\le d$, then  $\Sigma_{G,d}$  is supported on the interval $\left[-\sqrt d,\sqrt d\right]$.

\begin{Pro}\label{funny}
 Let
$(G_n,d_n)$ be an $\alpha$-regular sequence of large girth, such that $d_n\to \infty$. Let $\Sigma_n=\Sigma_{G_n, d_n}$.
Then, for each $k\ge0$, \begin{equation}\label{scmoment}\int_{-\sqrt d}^{\sqrt d}x^k\mathrm d\Sigma_n\to \int_{-2}^2\left(\sqrt\alpha x\right)^kw(x)\mathrm d x\end{equation} as $n\to\infty$. Thus, $\Sigma_n$ converges weakly to the semicircle distribution on the interval $\left[-2\sqrt \alpha , 2\sqrt \alpha \right]$.
\end{Pro}

Note that the limit in \eqref{scmoment} is 0 for $k$ odd and is $\alpha^{k/2}$ times the Catalan number $\frac1{k/2+1}\binom k{k/2}$ for $k$ even.

\begin{proof}We have $$\int_{-\sqrt d}^{\sqrt d}x^k\mathrm d\Sigma_n
=\frac1{\vv(G_n)}\sum \left(\frac\la{\sqrt{d_n}}\right)^k=\frac{\hom(C_k,G_n)}{\vv(G_n)d_n^{k/2}}.$$ For $n\ge n_0(k)$, all walks in $G_n$ of length $k$ are tree-like, whence
  $$\int_{-\sqrt d}^{\sqrt d}x^k\mathrm d\Sigma_n(x)=\int_{-2}^2x^k\mathrm d\rho_n(x)\to \int_{-2}^2\left(\sqrt\alpha x\right)^kw(x)\mathrm d x$$ as $n\to\infty$, by  Theorem~\ref{wc}. To deduce the weak convergence, we use that the semicircle measure is compactly supported.
\end{proof}
A different proof is possible based on the fact that Kesten--McKay measures converge to the semicircle law.

For random graphs, results similar to Proposition~\ref{funny} have been proved by Dumitriu and Pal~\cite{DP} and by Tran, Vu and Wang~\cite{TVW}. Those results are of course much deeper than Proposition~\ref{funny}.

Proposition~\ref{funny} fails for large \emph{essential} girth, even if each $G_n$ is exactly $d_n$-regular.  Indeed, for the hypercube sequence  $(Q_d,d)$, the measure $\Sigma_d$ is the (binomial) distribution of $$(X_1+\dots +X_d)/\sqrt d,$$ where the $X_i$ are i.i.d.\ random variables with $\mathbb P(X_1=1)=\mathbb P(X_1=-1)=1/2$, see \cite[Exercise 11.9]{L'}. Thus, $\Sigma_d$ converges weakly to the standard Gaussian distribution and not to the semicircle distribution, therefore its moments do not all converge to those of the semicircle law.

\section{Graphonings}\label{limobj}
 We propose a common generalization of graphons and graphings.
\begin{Def} A \it graphoning \rm is a  tuple $\mathbf G=(X,\mathcal B,\la,\mu, W)$, where $(X,\mathcal B, \la)$ is a probability space, $\mu:\mathcal B\to [0,\infty]$ is a measure, and $W:X^2\to [0,1]$ is a symmetric $(\mathcal B \otimes\mathcal B)$-measurable function such that \begin{itemize}
\item (degree bound)
$$\deg _X(x)\overset{\mathrm{def}}=\int_X W(x,y)\mathrm d\mu(y)\le 1$$ for all $x\in X$,

\item (degree measurability)
\begin{equation}\label{degdef}\deg_A (x)\overset{\mathrm{def}}=\int_A W(x,y)\mathrm d\mu(y)\end{equation} is a measurable function of  $x\in X$ for all $A\in\mathcal B$,
 and
\item (measure preserving property)
\begin{equation}\label{measurepres}\int_A\deg _B 
\mathrm d\la 
=\int_B\deg_ A 
\mathrm d\la
\end{equation} for all $A,B\in\mathcal B$.
\end{itemize}
\end{Def}

A graphoning with $\la=\mu$ is the same thing as  a graphon, except that graphons are measurable only w.r.t.\ the completion of $\mathcal B\otimes\mathcal B$, and thus their degrees are only almost measurable. A graphoning on a  Borel probability space  $(X,\mathcal B, \la)$, such that $\mu$ is the counting measure divided by $d$, and $W$ only takes values in $\{0,1\}$, is the same thing as a graphing.

For these two special cases of graphonings, it is known that degree measurability follows from the degree bound condition.
It is unclear to the author whether this holds for general graphonings, maybe under the assumption that the $\sigma$-algebra $\mathcal B$ is Borel. 

Note that $\mu$ is not in general $\sigma$-finite, so the Fubini Theorem is not applicable to the iterated integrals in  \eqref{measurepres}.

\subsection{Sub-Markov kernels and rooted homomorphism densities}
We wish to generalize the homomorphism densities of graphons that play a fundamental role in the limit theory of dense graphs developed by L\'aszl\'o Lov\'asz and his coauthors~\cite{BCLSV, L, LSz}. Technical difficulties are caused by the lack of the Fubini Theorem, but these can be dealt with. We treat rooted homomorphism densities first. Even this requires some preparation. It will save work later on if we introduce structures even more general than graphonings.
For this, let us recall a basic concept from the theory of Markov chains.
\begin{Def} A \it sub-Markov kernel \rm on a measurable space $(X,\mathcal B)$ is a  function \begin{equation*}
\begin{aligned}
\deg:   X  \times  \mathcal B & \to  [0,1]
\\
 ( x  ,  A) &\mapsto  \deg_A(x)
 \end{aligned}
 \end{equation*} such that the function $\deg_A:x\mapsto \deg_A (x)$ is measurable for all $A\in\mathcal B$ and the set function $\deg (x):A\mapsto\deg_A(x)$ is a measure for all $x\in X$.
 \end{Def}


Clearly, the degree function of a graphoning is a sub-Markov kernel. The measurability of $\deg_A$ implies its seemingly stronger form below. This is probably well known but we prove it to be self-contained.

\begin{Lemma}\label{stronger} Consider a  sub-Markov kernel $\deg$ on the measurable space
$(X,\mathcal B)$. 
 If $(Z,\mathcal C)$ is a measurable space and $f:X\times X\times Z\to [0,1]$ is measurable, then
$$\deg_f (x,z)\overset{\mathrm{def}}=\int_X f(x,y,z)\mathrm d\deg(x)(y)$$ is a measurable function of  $(x,z)\in X\times Z$, and takes values only in $[0,1]$.


\end{Lemma}

\begin{proof}  We have \[0\le\deg_f(x,z)\le \deg_X(x)\le 1\] for all $x$ and $z$.

The function $\deg_f$ is measurable, by the definition of a sub-Markov kernel, when $f$
 is the indicator of a direct product of measurable sets. By linearity, it is measurable when $f$ is the  indicator of a finite union of such products. By
  the Monotone Convergence Theorem, it follows that the set of measurable functions $f:X\times X\times Z\to [0,1]$ such that $\deg_f$ is measurable is closed under monotone pointwise limits and therefore contains all measurable indicator functions, thus all
 measurable stepfunctions.

 Any measurable function to $[0,1]$ can be uniformly approximated by measurable stepfunctions. An error with uniform upper bound $\epsilon$ in $f$ leads to an error with uniform upper bound $\epsilon$ in $\deg_f(x,z)$. This proves the Lemma because a  uniform limit of measurable functions is measurable.
\end{proof}

\begin{Cor}\label{degf}
If $\deg$ is a sub-Markov kernel on $(X,\mathcal B)$ and $f:X\to [0,1]$ is measurable, then the function $\deg_f:X\to[0,1]$ defined by $$\deg_f(x)\overset{\mathrm{def}}=\int_X f\mathrm d\deg(x)$$ is measurable and takes values only in $[0,1]$.
\end{Cor}

\begin{proof} Use Lemma~\ref{stronger} for $F(x,y,z)=f(y)$, with $Z=\{z\}$ being a  single point.
\end{proof}


\begin{Def}\label{compatible} A sub-Markov kernel $\deg$ on a measurable space $(X,\mathcal B)$ is \it compatible \rm with a
$(\mathcal B\otimes \mathcal B)$-measurable function $W:X^2\to [0,1]$ if
\begin{equation}\label{coinc} \int_AW(x_2,y)\mathrm d\deg(x_1)(y)=\int _A
W(x_1,y)\mathrm d\deg(x_2)(y)\end{equation}
 for all $x_1,x_2\in X$ and all $A\in\mathcal B$. 
\end{Def}

\begin{Lemma}
In a graphoning, $\deg$ is compatible with $W$.
\end{Lemma}

\begin{proof}
Both sides of \eqref{coinc} equal $$\int_AW(x_1,y)W(x_2,y)\mathrm d\mu(y).$$ Indeed, 
this is a special case of the well-known formula $$\int_A f \frac {\mathrm d \nu}{\mathrm d \mu} \mathrm d \mu=\int_A f\mathrm d\nu$$ involving a Radon-Nikodym derivative.
\end{proof}




To define rooted homomorphism densities, we will have to introduce labelings on the test graphs $F$. The compatibility discussed above will ensure that the density is independent of the labeling chosen.

\begin{Def} An \it admissible  labeling \rm of a connected  graph $F$ is a bijection $$\phi: V(F)\to \{0,1,\dots, \vv (F)-1\}$$  such that for all $1\le i\le \vv (F)$, the nodes with labels less than $i$ span a  connected subgraph.  Two admissible labelings are \it  adjacent \rm if a transposition $(i-1,i)$ of labels takes one to the other. This turns the set of admissible labelings of $F$ into a graph.
\end{Def}

\begin{Lemma}\label{conn}
\begin{itemize}
\item[(a)]
For any connected graph $F$, the graph of admissible labelings is connected.

\item[(b)] The admissible labelings such that a fixed node $o$ gets label 0 span a connected subgraph.
\end{itemize}\end{Lemma}

\begin{proof} (b)  Consider two admissible labelings $\phi$ and $\psi$ such that $\phi(o)=\psi(o)=0$. We prove that they are connected by a path.
We use induction on the number of inversions between them, i.e., the number of pairs $x,y\in V(F)$ such that $$(\phi(x)-\phi(y))(\psi(x)-\psi(y))<0.$$
If there are no inversions, then $\phi=\psi$. If there are inversions, then there are nodes $x$ and $y$ such that $\phi(x)=i$, $\phi(y)=i+1$, $\psi(x)>\psi(y)$.  Choose such $x$ and $y$ so that $i$ is largest possible.   Since $\psi$ is admissible, there is an edge in $F$ from $y$ to  a node $z$ with $\psi(z)<\psi(y)$ and therefore $\phi(z)<i$. Thus, composing $\phi$ with the transposition $(i,i+1)$ yields an admissible labeling $\phi'$ that has less inversions when compared to $\psi$ than $\phi$ does.

(a) We may assume that $F$ has at least two nodes. It suffices to show that for any two adjacent nodes $x$ and $y$ in $F$, there exist adjacent admissible labelings $\phi$ and $\psi$ such that $\phi(x)=\psi(y)=0$. Let $\phi$ be an admissible labeling with $\phi(x)=0$ and $\phi(y)=1$. Let $\psi$ arise from $\phi$ by swapping the labels of $x$ and $y$. Then $\psi$ is admissible and adjacent to $\phi$.
\end{proof}

\begin{Def}\label{rootedhomomdef} Let $\mathbf G=(X,\mathcal B, W, \deg)$ be a measurable space endowed with a symmetric measurable function $W:X^2\to[0,1]$ and a sub-Markov kernel $\deg$ that is compatible with $W$. 
 Let $x_0\in X$. Let $(F,o)$ be a connected  rooted graph. Fix any admissible labeling of $V(F)$ such that $o$ gets label 0. For any label $i=1,\dots, \vv(F)-1$, let $j(i)$ be  a label such that $j(i)<i$ and $j$ is adjacent to $i$ in $F$. Note that $j(1)=0$. Let $T$ be the spanning tree consisting of the edges $(i,j(i))$. We define the \it rooted homomorphism density \rm
\begin{equation*}\begin{aligned}t((F,o),(X,x_0))&=\\=\int_X\cdots\int_X&\prod_{kl\in E(F)-E(T)}W(x_k,x_l)
\mathrm d \deg (x_{j(\vv(F)-1)})(x_{\vv(F)-1})\cdots\mathrm d\deg (x_{j(1)})(x_1).\end{aligned}\end{equation*}
\end{Def}



\begin{Pro}\label{rootedwelldef}
The rooted homomorphism density  \begin{itemize}\item[(a)] is well defined, is in $[0,1]$, is measurable as a function of $x_0$, and
\item[(b)]  is independent of the admissible  labeling and the function $j$ chosen.
\item[(c)] If $F$ is a tree, then it is also independent of the function $W$.
\end{itemize}
\end{Pro}

\begin{proof}(a) By repeated application of Proposition~\ref{stronger}, we see that each integration yields a measurable function of the remaining variables, with values in $[0,1]$.

\medskip

(b) For a  given admissible labeling, the rooted homomorphism density does not depend on the function $j$ because of the condition~\eqref{coinc}.

 Let us assume that $V(F)=\{0,1,\dots, \vv(F)-1\}$, and the identity as well as the transposition $(i-1,i)$ are admissible labelings, where $i\ge 2$ is fixed. Then 
$j(i)<i-1 $ and 
we may  apply the Fubini Theorem to swap the two factors $$\mathrm d\deg(x_{j(i)})(x_i)\mathrm d\deg(x_{j(i-1)})(x_{i-1}),$$ showing that the two admissible labelings in consideration define the same value of the rooted homomorphism density.

An application of Lemma~\ref{conn}(b) finishes the proof.

\medskip

(c)
The product in Definition~\ref{rootedhomomdef} is empty if $F=T$.
\end{proof}

Definition~\ref{rootedhomomdef} may be frightening, but it becomes much nicer for graphonings.
From now on, we abbreviate $\mathrm d\mu(x_i)$ to $\mathrm dx_i$.

\begin{Rem}\label{rooted}
Consider a graphoning $\mathbf G=(X,\mathcal B,\la,\mu, W)$ with a specified point $x_0\in X$. Let $(F,o)$ be a connected rooted graph. Then we have
$$t((F,o),(\mathbf G,x_0))=\int_X\cdots\int_X\prod_{ij\in E(F)}W(x_i,x_j)\mathrm dx_{\vv(F)-1}\cdots\mathrm dx_1$$ if $V(F)$ is admissibly labeled by 0, 1, \dots, $\vv(F)-1$ so that $o$ gets label 0. Note that the Fubini theorem is not directly applicable to the right hand side of this formula because $\mu$ is not in general $\sigma$-finite. Note also that $\la$ plays no role here.
\end{Rem}

\subsection{Reversible kernels and unrooted homomorphism densities}
To define unrooted homomorphism densities, we will need the measure preserving property~\eqref{measurepres}. Again it is worthwhile to generalize this first. We recall another basic concept from Markov chain theory.
\begin{Def}\label{pseudo-graphoning}
A sub-Markov kernel $\deg$ 
  on a  probability space $(X,\mathcal B, \la)$ 
   is \it reversible   \rm w.r.t.\ $\la$
  \rm  if   the measure preserving condition~\eqref{measurepres} holds.
\end{Def}

In particular, the degree function of a graphoning is reversible.

On a measurable space  $(X,\mathcal B)$, there  can be many probability measures  that make a  given sub-Markov  kernel $\deg$ reversible. We call such measures $\la$  \it involution-invariant  \rm w.r.t.\ $\deg$ \rm because if we choose a  $\la$-random point $a\in X$ and then a  point $b\in X$ with conditional (sub-probability) distribution $\deg(a)$, then the pairs $(a,b)$ and $(b,a)$ have the same (sub-probability) distribution.
Indeed,  \eqref{measurepres} precisely means  the equality of these two measures on measurable product sets $A\times B\subseteq X^2$, and this implies equality on  the entire  $\sigma$-algebra $\mathcal B\otimes\mathcal B$. This implies the following well-known, crucial fact.

\begin{Lemma}\label{crucial}
If $\deg$ is a reversible sub-Markov kernel on a probability space $(X,\mathcal B, \la)$, and $f:X^2\to[0,1]$ is measurable, then  $$\int_X\int_X (f(x,y)-f(y,x))\mathrm d\deg (x)(y)\mathrm d\la (x)=0.$$
\end{Lemma}

\begin{Cor}\label{pseudostronger}
If $\deg$ is a reversible sub-Markov kernel on $(X,\mathcal B, \la)$ 
 and $f,g:X\to [0,1]$ are measurable functions, then \begin{equation}\label{fg}\int_X f\cdot \deg_g \mathrm d\la=\int_X g\cdot \deg_f \mathrm d\la.\end{equation}
\end{Cor}

\begin{proof} Use Lemma~\ref{crucial} for $F(x,y)=f(x)g(y)$.
\end{proof}

\begin{Cor}\label{remainsreversible}
If $\deg$ is a reversible sub-Markov kernel on $(X,\mathcal B, \la)$ 
 and $f:X^2\to [0,1]$ is a symmetric   $(\mathcal B\otimes\mathcal B)$-measurable function, then the sub-Markov kernel $f\deg$ defined by $$(f\deg)_A(x)=\int_A f(x,y)\mathrm d \deg(x)(y)$$ is again reversible.
 \end{Cor}

\begin{proof} The fact that $f\deg$ is a sub-Markov kernel follows from Lemma~\ref{stronger}. For reversibility, we need to show that the value $$\int_A (f\deg)_B\mathrm d\la=\int_A\int_B f(x,y)\mathrm d \deg(x)(y)\mathrm d\la (x)$$ is symmetric w.r.t\ $A$ and $B$. This is Lemma~\ref{crucial} for $F(x,y)=\mathbb 1_A(x)f(x,y)\mathbb 1_B(y)$.
\end{proof}


\begin{Def}\label{gengra} A \it pseudo-graphoning \rm is a probability space $(X,\mathcal B, \la)$ endowed with a symmetric $(\mathcal B\otimes\mathcal B)$-measurable function $W:X^2\to[0,1]$ and a reversible sub-Markov kernel $\deg$ that is compatible with $W$.
\end{Def}
Every graphoning is also a pseudo-graphoning. A pseudo-graphoning is a graphoning if and only if there exists a measure $\mu:\mathcal B\to [0,\infty]$ such that the equality~\eqref{degdef} holds for all $x\in X$ and $A\in\mathcal B$.

\begin{Pro} If $\mathbf G=(X,\mathcal B, \la, W, \deg)$ is a pseudo-graphoning and $f:X^2\to [0,1]$ is
symmetric and  $(\mathcal B\otimes\mathcal B)$-measurable, then
$f\mathbf G=(X,\mathcal B, \la, fW, f\deg)$ is also a  pseudo-graphoning.
\end{Pro}

\begin{proof}
The function $fW$ is symmetric and measurable. By Corollary~\ref{remainsreversible},
$f\deg$ is a reversible sub-Markov kernel. It remains to check that $f\deg$ is compatible with $fW$, which is trivial.
\end{proof}

\begin{Cor} If $\mathbf G=(X,\mathcal B, \la,   \mu, W)$ is a graphoning and $f:X^2\to [0,1]$ is
symmetric and  $(\mathcal B\otimes\mathcal B)$-measurable, then
$f\mathbf G=(X,\mathcal B, \la,  \mu, fW)$ is also a graphoning.
\end{Cor}

\begin{proof}
We have $$\int_A(fW)(x,y)\mathrm d\mu (y)=(f\deg)_A(x),$$
so the claim follows from the previous Proposition.
\end{proof}
This is a generalization of \cite[Lemma 18.19]{L} from L\'aszl\'o Lov\'asz's monograph: a  Borel subgraph of a graphing is a graphing.

Using reversibility, we can define unrooted homomorphism densities.

\begin{Def}\label{pseudot}
Consider a pseudo-graphoning 
$\mathbf G=(X,\mathcal B,\la,W,\deg)$. Let $F$ be a connected graph. 
 Choose  a root $o$ in $F$. Choose $x_0\in X$ randomly with distribution $\la$. We define the  \it homomorphism density \rm
\begin{equation}\label{unrootedformula}t(F,\mathbf G)=\mathbb Et((F,o),(\mathbf G,x_0))=\int_Xt((F,o),(\mathbf G,x_0))\mathrm d\la (x_0).\end{equation}
\end{Def}
Since the rooted homomorphism density is a measurable function of $x_0$ and takes values in $[0,1]$ only, the expectation above exists and is in $[0,1]$.


\begin{Pro}\label{unrootedwelldef}\begin{itemize}\item[(a)]
The homomorphism density $t(F,\mathbf G)$ 
 is independent of the root $o$.
 \item[(b)] If $F$ is a tree, then it is also independent of the function $W$.
 \end{itemize}
\end{Pro}

\begin{proof}  (a)   Given two adjacent nodes $o_0$ and $o_1$ in $F$,
consider an admissible labeling such that $o_0$ and $o_1$ get labels 0 and 1 respectively. 
For each $i\ge 2$, let $ j(i)<i$ be such that the nodes with labels $i$ and $j(i)$ are adjacent, and let $T$ be the spanning tree given by
the edges $(i,j(i))$ and $(01)$.
Consider the birooted homomorphism density $$f(x_0,x_1)=\int_X\cdots\int_X\prod_{kl\in E(F)-E(T)}W(x_k,x_l)
\mathrm d \deg (x_{j(\vv(F)-1)})(x_{\vv(F)-1})\cdots\mathrm d\deg (x_{j(2)})(x_2).$$ This does not depend on the function $j$ chosen because $\deg$ is compatible with $W$. We have $$t((F,o_0), (\mathbf G, x_0))=\int_Xf(x_0,x_1)\mathrm d\deg (x_0)(x_1)$$ and
 $$t((F,o_1), (\mathbf G, x_0))=\int_Xf(x_1,x_0)\mathrm d\deg (x_0)(x_1)$$ --- note that the labeling that arises by swapping the labels 0 and 1 is also admissible. These two rooted densities have the same expectation by Lemma~\ref{crucial}.

 \medskip

 (b) Immediate from Proposition~\ref{rootedwelldef}(c).
\end{proof}


\begin{Rem}\label{graphgraphoning} Let $F$ be a connected graph.

 For a graphoning $\mathbf G=(X,\mathcal B,\la,\mu, W)$,  we have
$$t(F,\mathbf G)=\int_X\int_X\cdots\int_X\prod_{kl\in E(F)}W(x_k,x_l)\mathrm dx_{\vv(F)-1}\cdots\mathrm dx_1\mathrm d\la (x_0)$$ if $V(F)$ is admissibly labeled by 0, 1, \dots, $\vv(F)-1$. Note again that the Fubini theorem is not directly applicable to  the right hand side of  this formula because $\mu$ is not in general $\sigma$-finite.

For  a graphon $\mathbf G$  --- which is  a graphoning with $\mu=\la$ ---  we recover the well-known homomorphism density
$$t(F,\mathbf G)=\int_{X ^ {V (F)}}\prod_{kl\in E(F)}W(x_k, x_l) \prod_{i\in V(F)}
\mathrm d\la (x_i).$$

For a graphing $\mathbf G$ --- which is  a graphoning with $\mu$ being $(1/d)$ times the counting measure ---  we recover a normalized version of the the well-known homomorphism frequency:
$$t(F,\mathbf G)=t^*(F, \mathbf G)/{d^{\vv(F)-1}},$$ where $$t^*(F,\mathbf G)=\int_{X}\hom ((F,o), (\mathbf G, x))\mathrm d\la (x).$$

For a graph $G$ with all degrees $\le d$, we can define  a graphoning as follows. Let $X=V(G)$ and  $\mathcal B=\mathcal P(X)$.
Let $\la$ be the uniform probability measure on $X$. Let $\mu=(\vv(G)/d)\la$.
Let $W:X^2\to\{0,1\}$ be the adjacency matrix of $G$.
This graphoning has the same (rooted and unrooted) homomorphism densities as $(G,d)$.
\end{Rem}

\subsection{Graph limits}
\begin{Def}
A \it limit \rm for a  convergent sequence $(G_n,d_n)$ is a pseudo-graphoning
$\mathbf G$ such that $t(F,G_n,d_n)\to t(F,\mathbf G)$ for all connected $F$. In this case, we write $(G_n,d_n)\to \mathbf G$. A \it true limit \rm is a limit which is a graphoning.
 \end{Def}

 In the rest of this paper, our main interest is in the existence of limits. Very little is known. We start with a very special example.

\begin{Prop}\label{Orbanz} Let $(G_n,d_n)$ be a convergent sequence such that $G_n$ is the disjoint union of graphs with $d_n$ nodes each. Then the sequence has  a  true limit.
\end{Prop}

\begin{proof} As explained in Example~\ref{smallcomp}, there exists a Borel probability measure $\gamma$ on the compact graphon space $\widetilde{\mathcal W}_0$, such that $$t(F,G_n,d_n)\to \int_{\widetilde{\mathcal W}_0} t(F,U)\mathrm d\gamma(U)$$ for all connected graphs $F$.

Let $\mathcal W_0$ be the space of labeled graphons endowed with the 1-norm --- not the cut norm, which is used to define the topology in $\widetilde{\mathcal W}_0$. I.e.,  $\mathcal W_0$ is the subset of the Banach space $\mathfrak L^1\left([0,1]^2\right)$
that consists of all symmetric functions with values in $[0,1]$. By \cite[Theorem 1]{OSz} of Orbanz and Szegedy, there exists a measurable map $\xi:
\widetilde{\mathcal W}_0\to {\mathcal W}_0$ which is  a section (one-sided inverse) of the canonical quotient map ${\mathcal W}_0\to  \widetilde{\mathcal W}_0$. Note that for each $U\in\widetilde{\mathcal W}_0$, the function $\xi(U)\in{\mathcal W}_0$ is defined only almost everywhere, but for each $f\in {\mathcal W}_0$, we may use a  variant of \cite[Definition 2.2]{AP} to choose a canonical representative which is defined everywhere: $$\bar f(x,y)=\limsup_{\epsilon\to 0}\frac{1}{4\epsilon^2}\int_{x-\epsilon}^{x+\epsilon}\int_{y-\epsilon}^{y+\epsilon}f\mathrm d\la_2,$$ where $\la_2$ stands for 2-dimensional Lebesgue measure, and undefined values of $f$ are taken to be zero. It is easy to see that the function $$\mathcal W_0\times [0,1]^2\to [0,1], \qquad (f, x, y)\mapsto \bar f(x,y)$$ is Borel measurable; this was observed by Viktor Kiss (unpublished). It follows that the function
$$\widetilde{\mathcal W}_0\times [0,1]^2\to [0,1], \qquad (W, x, y)\mapsto \overline{\xi(U)}(x,y)$$ is also Borel measurable.

Let $X=\widetilde{\mathcal W}_0\times [0,1]$ and define $ W:X^2\to [0,1]$ by putting $$W ((U,x),(V,y))=\mathbb{1} _{U=V}\overline{\xi(U)}(x,y).$$ The function $W$ is clearly symmetric and Borel measurable.

For all $A\subseteq X$, let
$$A_U=\left\{x\in [0,1]: (U,x)\in A\right\}\qquad \left(U\in \widetilde{\mathcal W}_0\right).$$ Let $\mathcal B\subset\mathcal P(X)$ be the $\sigma$-algebra of Borel sets. 
For all $A\in\mathcal B$,  define
$$\mu(A)=\sum_{U\in \widetilde{\mathcal W}_0} \la_1(A_U).$$
Let $\la=\gamma\times \la_1$,  where $\la_1$ stands for 1-dimensional Lebesgue measure.

Let $\mathbf G=(X,\mathcal B, \la, \mu, W)$. It is straightforward to check that $\mathbf G$ is a graphoning and $$t(F, \mathbf G)=\int_{\widetilde{\mathcal W}_0} t(F,U)\mathrm d\gamma(U)=\lim_{n\to\infty} t(F, G_n, d_n)$$ for all connected graphs $F$.
\end{proof}

 \subsection{Hausdorff  limits}
We now introduce special graphonings that involve geometric measure theory.
\begin{Def} A \it Hausdorff graphoning \rm is a  graphoning of the form $$\mathbf G=\left(X,\mathcal B, \la,\mu,W\right),$$ where $X$ is a metric space, $\mathcal B$ is the $\sigma$-algebra of Borel sets, $\la$ is 1-dimensional Hausdorff measure, and $\mu$ is a  Hausdorff measure with some gauge function $h$.
I.e., $h\ge 0$ is a  right-continuous
nondecreasing function on a right neighborhood of 0 and
$$\mu(B)=\lim_{\delta\to 0} \inf \left\{\sum_{i=1}^\infty h(\diam I_i) : \diam I_i<\delta\; \textrm{for all $i$, and}\; B\subseteq \bigcup
_{i=1}^\infty I_i\right\}$$ for any Borel set $B$.

A \it Euclidean graphoning \rm is  a  Hausdorff graphoning where $X=[0,1]$ with the Euclidean metric.
\end{Def}

Note that if gauge functions $h_1$ and $h_2$ satisfy $(1-\epsilon)h_1(x)\le h_2(x)\le(1+\epsilon)h_1(x)$
for $0\le x<\delta(\epsilon)$, then they define the same Hausdorff measure.

The gauge function $h(x)=x$ gives rise to the 1-dimensional Hausdorff measure. For $X=[0,1]$, this is Lebesgue measure; the corresponding Euclidean  graphonings are Borel measurable graphons. The constant gauge function $h(x)=1/d$ gives rise to the counting measure divided by $d$; in this case  $\{0,1\}$-valued Hausdorff  graphonings are graphings.

\begin{Def}  A \it Hausdorff \rm (resp.\ \it Euclidean\rm) \it limit \rm for a  convergent sequence $(G_n,d_n)$ is a limit which is  a Hausdorff  \rm (resp.\ Euclidean\rm)  graphoning with a gauge function $h$ such that $h(1/\vv(G_n))\sim 1/d_n$ as $n\to\infty$. 
\end{Def}

Recall from Definition~\ref{t} that the homomorphism density $t(F,G,d)$ involved a  normalization by an appropriate power of $d$ in order to be in $[0,1]$. The role of the gauge function $h$ is to encode in the limit object not only the limiting homomorphism densities, but also the growth rate of the degree bound $d_n$.

For  a convergent sequence $(G_n)$ of dense graphs with $\vv(G_n)=n$, a Euclidean  limit for the convergent sequence $(G_n, n)$ is the same thing as a limiting (Borel measurable) graphon on $[0,1]$. For a  Benjamini--Schramm convergent sequence $(G_n)$ with degree bound $d$ and  with $\vv(G_n)\to\infty$,
a $\{0,1\}$-valued Euclidean  limit for the convergent sequence $(G_n, d)$ is the same thing as a limiting  graphing on $[0,1]$.

\begin{Ex}\label{limprod}
The sequence $(G_n, d_n)$ of Example~\ref{dirprod}, provided that $ \vv(\Gamma_i)\ge 2$ for all $i$, always has  a Hausdorff limit such that in the underlying metric space, all nonzero distances are of the form $1/\vv(G_n)$, and $W$ is $\{0,1\}$-valued.
 \end{Ex}

 \begin{proof} Let $X=\prod_{i=1}^\infty V(\Gamma_i)$. The distance of two points in $x, y\in X$ is defined to be $1/\vv(G_n)$ if  $n+1=\inf\{i:x_i\ne y_i\}$.  The corresponding 1-dimensional Hausdorff measure $\la$ will be the product of the uniform probability measures $\la_i$ on $V(\Gamma_i)$. Set $h (1/\vv(G_n))=1/d_n$. This is well defined since  $\vv(G_n)<\vv(G_{n+1})$ for all $n$. The corresponding Hausdorff measure $\mu$  will be the product of the measures $\mu_i=(\vv(\Gamma_i)/\delta_i)\la_i$. Let $\mathbf G= (X,\mathcal B, \la, \mu, W)$, where $W(x,y)=1$ if $x_i$ and $y_i$ are adjacent in $\Gamma_i$ for all $i$, and $W(x,y)=0$ otherwise. This $\mathbf G$  is the direct product of the graphonings that correspond to the $(\Gamma_i, \delta_i)$ by the end of Remark~\ref{graphgraphoning}. It is easy to see that $\mathbf G$ is a Hausdorff limit 
of $(G_n, d_n)$.
\end{proof}

The author is unable to answer the fundamental
\begin{?}\label{problem} \begin{itemize}\item[(a)]
Which convergent sequences have (true, Hausdorff, Euclidean) limits?

\item[(b)]
Which pseudo-graphonings arise as (Hausdorff) limits?
\end{itemize}
\end{?}

In the dense case, the Euclidean (i.e., graphon) versions of both questions have been answered by L.\ Lov\'asz and B.\ Szegedy \cite{L, LSz}; the answer is ``all''. In the bounded degree
case, the 
 graphing version of (a) was solved by D.\ Aldous and R.\ Lyons \cite{AL} and by G.\ Elek \cite{E}, see also \cite[Theorem 18.37]{L}; the answer is ``all''; while the answer ``all'' for the graphing version of (b)  is the Aldous--Lyons Conjecture. (In our setting, we should say ``all simple graphings'' because we are only allowing simple graphs.)

\subsection{Acyclicity and regularity}
In the remaining part of this paper, our main focus is on constructing limit objects for convergent sequences of large essential girth.
First, we characterize when  the cycle densities of a graphoning vanish.

A sub-Markov kernel generates a sub-Markov chain in the usual way:
\begin{Def} Let $\deg$ be a  sub-Markov kernel on the measurable space $(X,\mathcal B)$. For $x\in X$, let $\deg^0(x):\mathcal B\to\{0,1\}$  be the Dirac measure at $x$. 
 Let $\deg^1=\deg$. If $i$ and $j$ are positive integers summing to $k$, then define \begin{equation}\label{degk}\deg^k_A(x)=\int_X\deg _A^i \mathrm d\deg ^j(x).\end{equation} This yields a well defined sub-Markov kernel $\deg^k$ on $(X,\mathcal B)$.  
\end{Def}

\begin{Pro} Let $k\ge 2$. For  a graphoning $\mathbf G$, the following are equivalent. \begin{itemize}
\item[(a)] $t(C_{k+1},\mathbf G)=0$;
\item[(b)] The neighborhood  $N(x)=\{y\in X: W(x,y)>0\}$ has $\deg^k(x)$-measure zero for $\la$-a.e.\ $x$;
\item[(c)] $\deg^k(x)\perp \deg(x)$  (singular measures) for $\la$-a.e.\  $x\in X$.\end{itemize}
\end{Pro}


\begin{proof} (a) $ \Leftrightarrow$ (b):
We have $$t(C_{k+1},\mathbf G)=\int_X\int_XW(x,y)\mathrm d \deg^k(x)(y)\mathrm d\la(x).$$ Statement (a) holds if and only if this is zero, i.e., $$\int_XW(x,y)\mathrm d \deg^k(x)(y)=0$$ for $\la$-a.e.\ $x$, which is equivalent to (b).

(b) $\Rightarrow$ (c): The measure $\deg(x)$ is concentrated on the set $N(x)$.

(c) $\Rightarrow$ (b):
   The formula~\eqref{degk} for $i=1$, together with the definition~\eqref{degdef} of $\deg$ in a graphoning,    show that $\deg^k(x)$ is absolutely continuous with respect to $\mu$, for all $x$, if $k\ge 1$. Assume that
$\deg^k(x)\perp \deg(x)$ for a fixed $x$; we prove that the set $N(x)$ has $\deg^k(x)$-measure zero.
The set $N(x)$ can be written as a union $A\cup B$, where $\deg^k_A(x)=\deg_B(x)=0$, because $X$ can be written as such a  union, by the definition of singular measures. By the definition of the measure $\deg(x)$, we have $\mu(B)=0$, whence $\deg^k_B(x)=0$ and therefore  $\deg^k_{N(x)}(x)=\deg^k_{A\cup B}(x)=0$ as claimed.
\end{proof}

\begin{Def}
Consider a probability space endowed with a  sub-Markov kernel:  $\mathbf G=(X,\mathcal B, \la, \deg)$. The space $\mathbf G$ (or the kernel $\deg$)  
 is \it acyclic \rm if $\deg^k(x)\perp \deg(x)$   for $\la$-a.e.\  $x\in X$ and all $0\le k\ne 1$.
\end{Def}

In particular, a graphoning is acyclic if and only if all cycle densities are zero. 

In the next subsection, we will be interested in limits of \it regular \rm sequences (of large essential girth). We now introduce the corresponding limit objects.

\begin{Def} Let $0\le\alpha\le 1$. 
 Consider a probability space endowed with a  sub-Markov kernel:  $\mathbf G=(X,\mathcal B, \la, \deg)$. The space $\mathbf G$ (or the kernel $\deg$)  
 is
 \it  $\alpha$-regular \rm if for all $k\ge 0$, and for $\la$-a.e.\ $x\in X$, we have $\deg_X(y) =\alpha$ for $\deg^k(x)$-a.e. $y\in X$.
\end{Def}

In particular, a  Markov kernel is 1-regular.

Regular kernels can be characterized in terms of homomorphism densities of rooted trees. Note that rooted tree densities as in Definition~\ref{rootedhomomdef}   depend neither on the function $W$ --- cf.\ Proposition~\ref{rootedwelldef}(c) ---,  nor on the probability measure $\la$,  therefore rooted tree densities of a measurable space  endowed with a  sub-Markov kernel make sense.

\begin{Pro}\label{alpharegrooted}
 Consider a probability space endowed with a  sub-Markov kernel:  $\mathbf G=(X,\mathcal B, \la, \deg)$. 
 The following are equivalent.
\begin{itemize}
\item[(a)] The space $\mathbf G$ is $\alpha$-regular.


\item[(b)] For $\la$-a.e.\ $x\in X$, we have  $$t((P_{k+2}, o), (\mathbf G, x))= \alpha^{k+1}$$ 
 and  $$t((D_{k+3}, o), (\mathbf G,x))= \alpha^{k+2}$$ for all $k\ge 0$, where $o$ is a  leaf 
 (farthest from the trivalent node in the case of $D_{k+3}$, $k\ge 1$), except in $D_3$, where $o$ is the non-leaf.


\item[(c)] For all rooted trees $(F, o)$, we have  $t((F,o), (\mathbf G, x))= \alpha^{\e(F)}$ for $\la$-a.e.\ $x\in X$.
\end{itemize}
\end{Pro}

\begin{proof}
Assuming (a), we easily get (c) by induction on $\vv(F)$. The implication (c)  $\Rightarrow$  (b) is trivial. 
Assuming  (b), we prove (a).
We have $$\alpha^{k}=t((P_{k+1},o),(\mathbf G,x))=
\int_X\mathrm d\deg^k(x)(y)
,$$ $$\alpha^{k+1}=t((P_{k+2},o),(\mathbf G,x))=
\int_X\deg_X(y)\mathrm d\deg^k(x)(y)
,$$  and
$$\alpha^{k+2}=t((D_{k+3},o),(\mathbf G,x))=
\int_X(\deg_X(y))^2\mathrm d\deg^k(x)(y)
$$ for all $k\ge 0$ and $\la$-a.e.\ $x\in X$; note that $P_1\simeq K_1$.
From the condition of equality in the Cauchy--Schwarz inequality, we see that for all $k$, there exists an $\alpha_k$ such that  for $\la$-a.e.\ $x\in X$, we have $\deg_X(y) =\alpha_k$ for $\deg^k(x)$-a.e. $y\in X$. Then $$\alpha_0\cdots\alpha_k=t((P_{k+2},o), (\mathbf G,x))=\alpha^{k+1}$$ for all $k\ge 0$ and $\la$-a.e.\ $x\in X$. If $\alpha>0$, then this implies that $\alpha_k=\alpha$ for all $k$, and
$\mathbf G$ is $\alpha$-regular. If $\alpha=0$, then  
we get $\alpha_0=0$, i.e., $\deg_X(x)=0$ for $\la$-a.e.\ $x\in X$.  But then $\deg^k(x)=0$ for $\la$-a.e.\ $x\in X$ and all $k\ge 1$, and therefore $\mathbf G$ is 0-regular.
\end{proof}

For reversible kernels, the  characterization of regularity becomes much nicer.

\begin{Lemma} Let $0\le\alpha\le 1$. 
 Consider a probability space endowed with a reversible sub-Markov kernel:  $\mathbf G=(X,\mathcal B, \la, \deg)$. The space $\mathbf G$ 
 is
  $\alpha$-regular if and only if  for $\la$-a.e.\ $x\in X$, we have $\deg_X(x) =\alpha$.
\end{Lemma}

\begin{proof} If $\mathbf G$ 
 is
  $\alpha$-regular, then for $\la$-a.e.\ $x\in X$, we have $\deg_X(y)=\alpha$ for $\deg^0(x)$-a.e.\ $y\in X$. But $\deg^0$ is Dirac measure at $x$, so we have $\deg_X(x)=\alpha$ for $\la$-a.e. $x\in X$, as claimed.

  Conversely, assume that the set $A=\{y\in X:\deg_X(y)\ne \alpha\}$ has $\la(A)=0$. What we want to  prove is that $\deg^k_A(x)=0$ for all $k\ge 0$ and $\la$-a.e.\ $x\in X$. Let $A_k=\{x\in X:\deg_A^k(x)>0\}$. We use induction on $k$ to show that $\la(A_k)=0$. For $k=0$, this holds because $A_0=A$. 
  If it holds for $k-1$, then it also holds for $k$ because $$\deg^k_A(x)=\int_X\deg_A^{k-1}\mathrm d\deg (x)=\int_{A_{k-1}}\deg_A^{k-1}\mathrm d\deg (x)=0$$ for $\la$-a.e. $x\in X$.  Indeed, $\deg_{A_{k-1}}(x)=0$ for $\la$-a.e. $x\in X$ because $$\int_X\deg_{A_{k-1}}\mathrm d\la=\int_{A_{k-1}}\deg_X\mathrm d\la=0$$ by reversibility of the kernel $\deg$ and by the induction hypothesis.
\end{proof}

Regular reversible kernels can be characterized in terms of homomorphism densities of trees. Recall from Proposition~\ref{unrootedwelldef}(b)  that tree densities of a  pseudo-graphoning do not depend on the function $W$, therefore tree densities of a probability space  endowed with a reversible sub-Markov kernel make sense.

\begin{Pro}\label{alphareggraphoning}
 Consider a probability space endowed with a reversible sub-Markov kernel:  $\mathbf G=(X,\mathcal B, \la, \deg)$. 
 The following are equivalent.
\begin{itemize}
\item[(a)] The space $\mathbf G$ is $\alpha$-regular.


\item[(b)] We have  $t(K_{2}, \mathbf G)= \alpha$  and  $t(P_{3}, \mathbf G)= \alpha^{2}$.

\item[(c)] For all  trees $F$, we have  $t(F, \mathbf G)=\alpha^{\e(F)}$.

\end{itemize}
\end{Pro}

\begin{proof}
Assuming (a), we  get (c) from Proposition~\ref{alpharegrooted}(c). The implication  (c)  $\Rightarrow$  (b) is  trivial. 
Assuming  (b), we prove (a). The degree of a $\la$-random point $x\in X$ has expectation $t(K_2,\mathbf G)=\alpha$ and variance $t(P_3,\mathbf G)-t(K_2,\mathbf G)^2=\alpha^2-\alpha^2=0$, therefore it is a.s.\ $\alpha$.
\end{proof}







\subsection{Hausdorff limits of regular sequences of large essential girth}\label{limcubeproj}

\begin{Lemma}\label{constr} If  $0\le \alpha\le 1$, and  $h\ge 0$ is a continuous   non-decreasing function on a right neighborhood of 0,
such that $h(0)=0$ but $h(x)/x\to\infty$ as $x\to 0$, then
\begin{itemize}
\item[(a)] there exists an $\alpha$-regular  acyclic Hausdorff graphoning $\mathbf G$ with gauge function $h$, such that $W$ is $\{0,1\}$-valued.
\item[(b)] If, in addition, the gauge function $h$ is concave, then $\mathbf G$
can be chosen to be Euclidean.
\end{itemize}
\end{Lemma}

\begin{proof}
If $\alpha=0$, let $X=[0,1]$ with the Euclidean metric, and let $W=0$ everywhere. This is a 0-regular acyclic Euclidean graphoning with gauge function $h$.

If $\alpha>0$, then we may, and do, assume that $\alpha=1$, since we may replace $h$ by $h/\alpha$.

\medskip

(a) For $i=1,2,\dots $,  choose positive integers $\gamma_i$ and $\delta_i$ such that $\delta_i$ is even and  $\gamma_i>\delta_i$ for all $i$, $\delta_i$ and $\gamma_i/\delta_i$ both tend to $\infty$ as $i\to\infty$, and $h(1/(\gamma_1\cdots\gamma_n))\sim 1/(\delta_1\cdots \delta_n)$ as $n\to\infty$.
Let $V(\Gamma_i)=\mathbb Z/\gamma_i\mathbb Z$, and join two nodes by an edge if their distance is $\le\delta_i/2$ to get a $\delta_i$-regular graph $\Gamma_i$. For all $k\ge 3$, we have $$\limsup_{i\to\infty} t(C_k,\Gamma_i, \delta_i)\le 3/4,$$ so the 1-regular sequence $(G_n, d_n)$, where $G_n=\Gamma_1\times \cdots\times \Gamma_n$ and $d_n=\delta_1\cdots\delta_n$, has large essential girth. The claim now follows from Example~\ref{limprod}.

\medskip

(b) For $i=1,2,\dots $,  choose  integers $\gamma_i\ge \delta_i>1$ such that $\delta_i-1|\gamma_i-1$ for all $i$, $\delta_i$ and $\gamma_i/\delta_i$ both tend to $\infty$ as $i\to\infty$, and $$|\delta_1\cdots \delta_nh(1/(\gamma_1\cdots\gamma_n))- 1|<1/2^n$$ for all $n$.
Let \begin{equation*}
S=\left\{\sum_{n=1}^\infty \frac {a_n}{\gamma_1\cdots\gamma_n}\; :\; 0\le a_n<\gamma_n,
\; a_n \equiv 0\mod (\gamma_n-1)/(\delta_n-1)\right\}.\end{equation*}
I.e., $S\subset [0,1]$ is the set of numbers which, in the mixed radix system with base $\gamma_1$, $\gamma_2$, \dots, have a representation such that the $n$-th digit is divisible  by $(\gamma_n-1)/(\delta_n-1)$ for all $n$. In other words, $S=\bigcap_{n=0}^\infty S_n$, where $$S_n=\bigcup_{r\in \mathcal R_n
}
I_r,$$ $\mathcal R_n$ is the set of integer sequences  $(r_1,\dots, r_n)$ such that $0\le r_i\le \delta_i-1$ for all $i$, $I_{\emptyset}=[0,1]$, and $I_r$ is a compact interval of length $1/(\gamma_1\cdots \gamma_n)$, 
  such that the two intervals $I_{r_1,\dots, r_{n-1}, 0}$ and $I_{r_1,\dots, r_{n-1},\delta_n-1}$ share a left, resp.\ right endpoint with $I_{r_1,\dots, r_{n-1}}$, and the midpoints of the $\delta_n$ intervals  $I_{r_1,\dots, r_{n-1}, 0}$, \dots,  $I_{r_1,\dots, r_{n-1},\delta_n-1}$ form an arithmetic progression.

 Since each $I_r$ is compact, so is $S_n$, and therefore so is  $S$.

Let $\mu$ be the Hausdorff measure with gauge function $h$. We shall now prove that $\mu (S)=1$.
This is closely related to \cite[Theorem 1]{kinai}. The basic idea is found already in  \cite[pp.\ 14--15]{Fa}.

 We have $S\subset S_n=\bigcup I_r,$ where \begin{equation*}\sum_{r
 } h(\diam I_r)=\delta_1\cdots\delta_n\cdot h(1/(\gamma_1\cdots\gamma_n))\to1\end{equation*} and $$\max_r\diam I_r=1/(\gamma_1\cdots\gamma_n)\to 0$$ as $n\to\infty$, whence $\mu(S)\le 1$.

For the converse inequality, assume that $S\subseteq\bigcup_{J\in\mathcal J} J$, where $\mathcal J$  is countable and each $\diam J$ is smaller than the length of a shortest component of $[0,1]-S_n$ for a  given $n$. We show that $$\sum_{J\in\mathcal J} h(\diam J)> 1-\frac1{2^{n-2}}.$$ 
We may assume (by taking convex hull and fattening a bit) that the sets $J$ are open intervals with endpoints not in $S$. By compactness, we may assume that there are only finitely many of them. Now  we may change our mind and assume (by cutting off superfluous bits) that each $J$ is  the convex hull of two intervals $I_r$,  where $r\in R_N$ with a fixed $N$, while each $J$ is contained in some $I_r$ with $r\in R_n$. We may also assume that the intervals $J\in\mathcal J$
are pairwise disjoint.

Let $\mathcal J'$ be the set of nonempty intervals arising by intersecting each $J\in\mathcal J$ with each connected component of $S_{n+1}$.

It suffices to show that $$\sum_{J'\in\mathcal J'} h(\diam J')-\sum _{J\in\mathcal J}h(\diam J)\le \sum_{r'\in\mathcal R_{n+1}} h(\diam I_{r'})-\sum _{r\in \mathcal R_n}h(\diam I_r),$$  because the right hand side is $$\delta_1\cdots \delta_n\delta_{n+1}h(1/(\gamma_1\cdots\gamma_n\gamma_{n+1}))-\delta_1\cdots \delta_nh(1/(\gamma_1\cdots\gamma_n))<3/2^{n+1}.$$

  Let $r\in \mathcal  R_n$ be fixed. It suffices to show that
   $$\sum_{J'\in\mathcal J', J'\subset I_r} h(\diam J')-\sum _{J\in\mathcal J, J\subseteq I_r}h(\diam J)\le \sum_{r_{n+1}=0}^{\delta_{n+1}-1} h(\diam I_{r, r_{n+1}})-h(\diam I_r),$$
  because summation upon $r$ gives our previous claim. The last inequality follows from the concavity of $h$. Indeed, if an in interval $J\in\mathcal J$ with  $J\subseteq I_r$ contains exactly $s$ of the $\delta_{n+1}-1$ connected components of $I_r\setminus S_{n+1}$, then
  $$\sum_{J'\in\mathcal J', J'\subseteq J} h(\diam J')-h(\diam J)\le\frac s{\delta_{n+1}-1}\left( \sum_{r_{n+1}=0}^{\delta_{n+1}-1} h(\diam I_{r, r_{n+1}})-h(\diam I_r)\right),$$
  and summation w.r.t.\ $J$ yields our previous inequality.
    This proves that $\mu(S)=1$.

For $x,y\in [0,1]$, put $W(x,y)=1$ if $|x-y|\in S$ and $W(x,y)=0$ otherwise. Let $\lambda $ be Lebesgue measure on $\mathcal B=\mathcal B[0,1]$.
We must prove that the tuple $$\mathbf G=([0,1], \mathcal B, \la,\mu, W)$$ is  a 1-regular Euclidean graphoning with gauge function $h$.
 Firstly, the function $W$ is semicontinuous and therefore Borel measurable. We have $\deg_{[0,1]}(x)=1$ for all $x\in [0,1]$. When $A$ is an interval, the function $\deg_A$ is continuous and therefore Borel measurable.
When $A$ is an open set, the function $\deg _A$ is  still Borel measurable because $A$ is a  countable disjoint union of intervals, thus $\deg_ A$ is Baire 1 (i.e., a pointwise limit of continuous functions).  The class of Borel subsets $A$ of $[0,1]$ such that $\deg_A$ is Borel measurable is closed under monotone sequential limits and contains all open sets, therefore contains all Borel sets, cf.\ \cite[Section II.6]{D}.

It remains to check the measure preserving property~\eqref{measurepres}. Observe that $W$ is the indicator of a set that is a union of lines with slope $45^\circ$ intersected with the unit square. Any union of such lines  is symmetric w.r.t.\ any line of slope $-45^\circ$. We deduce \eqref{measurepres} for intervals $A,B\subseteq [0,1]$ of equal length. Since any rectangle can be exhausted by squares, \eqref{measurepres} holds for any intervals $A$ and  $B$ by the Monotone Convergence Theorem.
For a  fixed interval $A$, both sides of \eqref{measurepres}, as functions of the Borel set $B$, are finite measures that coincide on intervals, so coincide on all Borel sets. For a fixed Borel set $B$, the two sides coincide on all intervals $A$, so on all Borel sets $A$.
This proves that $\mathbf G$ is indeed a 1-regular Euclidean graphoning with gauge function $h$.

We show that $\mathbf G$ is acyclic.
 Since $S$ is symmetric w.r.t.\ 1/2, this amounts to saying that for any $k\ge 2$, the modulo 1 sum of $k$ independent $\mu$-random elements of $S$
is a.s.\ not in $S$. But $\mu$-random means that for each $n$, we choose a  value from $\{0, 1,\dots, \delta_n-1\}$ uniformly and multiply it by $(\gamma_n-1)/(\delta_n-1)$  to get the $n$-th digit in the mixed radix expansion; and we do this independently for all $n$. There will a.s.\ be infinitely many indices $n$ such that there is  carrying from the $n$-th digit to  the previous digit when we perform the $k$-fold addition, but  there is no carrying from the $(n+1)$-th digit to the $n$-th. If such an $n$ is large enough, then in
 the $k$-fold modulo 1 sum the $n$-th digit $a_n$ is not divisible by the corresponding $(\gamma_n-1)/(\delta_n-1)$.
\end{proof}

\begin{Th}\label{reggirthlim} If  $0\le \alpha\le 1$, and $(G_n, d_n)$ is an $\alpha$-regular sequence of large essential girth,
such that $\vv(G_n)<\vv(G_{n+1})$ and $d_n\le d_{n+1}$ for all $n$, then
\begin{itemize}
\item[(a)] $(G_n,d_n)$ has a Hausdorff limit such that $W$ is $\{0,1\}$-valued.
\item[(b)] If, in addition, the function $1/\vv(G_n)\mapsto 1/d_n$ is concave, then $( G_n,d_n)$
has a Euclidean limit  such that $W$ is $\{0,1\}$-valued.
\end{itemize}
\end{Th}

\begin{proof}
Let $h(1/\vv(G_n))=1/d_n$, and let $h$ be linear on each of the intervals $$\left[1/\vv(G_{n+1}), 1/\vv(G_n)\right].$$ Put $h(0)=\lim (1/d_n)$  to make $h$ right-continuous. If $h(0)>0$, then $d_n$ stabilizes to a value $d$. Then $\alpha d$ must be an integer, and $G_n$ is Benjamini--Schramm convergent to the $\alpha d$-regular tree, which can be represented by a graphing on $[0,1]$.  From now on, we assume that $d_n\to\infty$, i.e., $h(0)=0$.

If $\alpha=0$, let $X=[0,1]$ with the Euclidean metric, and let $W=0$ everywhere. This is a Euclidean limit for $(G_n,d_n)$.

If $\alpha>0$, then, by Proposition~\ref{notdense}, we have $\vv(G_n)/d_n\to\infty$ as $n\to\infty$, and therefore $h(x)/x\to\infty$ as $x\to 0$.
The Theorem now follows from Lemma~\ref{constr}.
 \end{proof}

The rest of this subsection is not logically necessary, it is only to illustrate Lemma \ref{constr} and Theorem~\ref{reggirthlim}.
We work out two examples: we explicitly construct  Euclidean  limits of the sequence of hypercubes and the sequence of projective planes. Let $\mu_{\textrm{cube}}$ and $\mu_{\textrm{proj}}$ be the Hausdorff measures on $[0,1]$ corresponding to the gauge functions $$h_{\textrm{cube}}(x)=1/
\log_2 (1/x)
\qquad \textrm{and}\qquad h_{\textrm{proj}}=\sqrt{2x},$$ respectively. Note  that these gauge functions have the right growth rate: \begin{equation}\label{growth}h_{\mathrm{cube}}(1/\vv(Q_d))=1/d \qquad \mathrm{and} \qquad h_{\mathrm{proj}}(1/\vv(G))\sim 1/(q+1)\end{equation} if $G$ is the incidence graph of a  projective plane of order $q\to\infty$.
Observe also that $h_{\textrm{cube}}$
is concave on $[0,1/e]$ and $h_{\textrm{proj}}$
is concave on $[0,+\infty)$. This will help us to calculate the Hausdorff measures of carefully constructed sets.
The following construction relies on  a rather special property of these two functions $h$:  the numbers $1/h^{-1}(1/2)$ and $h^{-1}(1/2^n)/h^{-1}\left(1/2^{n+1}\right)$ $(n=1,2,\dots)$ are integral powers of $ 2$.
Thus, we can get away with binary expansions instead of the mixed radix expansions above, and the inequalities involved in the proof of Lemma~\ref{constr}(b) become much simpler.

Let \begin{equation}\label{Scube}
S_{\textrm{cube}}=\left\{\sum_{j=1}^\infty a_j2^{-j}:a_j\in\{0,1\},
a_1=a_2, a_3=a_4, a_5=\dots=a_8, a_9=\dots=a_{16}, \dots\right\}\end{equation}
and
\begin{equation}\label{Sproj}
S_{\textrm{proj}}=\left\{\sum_{j=1}^\infty a_j2^{-j}:a_j\in\{0,1\},
a_1=a_2= a_3,a_4= a_5, a_6= a_7, a_8=a_9,\dots\right\}.\end{equation}
\begin{Pro}\label{Hausdorff}\begin{itemize}
\item[(a)] The sets $S_{\mathrm{cube}}$ and  $S_{\mathrm{proj}}$ are  compact.

\item[(b)] $\mu_{\mathrm{cube}}(S_{\mathrm{cube}})=\mu_{\mathrm{proj}}(S_{\mathrm{proj}})=1$.\end{itemize}
\end{Pro}

For the `proj' case, this is well known \cite[page 15]{Fa};
it is also  a special case of \cite[Theorem 1]{kinai}. The `cube' case can be proved by the same technique, or a proof can be extracted from that of Lemma~\ref{constr}(b). We include a proof for the convenience of the reader.

\begin{proof} In both cases, we have $S=\bigcap_{n=0}^\infty S_n$, where $$S_n=\bigcup_{i\in\{0,1\}^n}
I_i,$$ and $I_i$ is a compact interval with $h(\diam I_i)=1/2^n$ for each $i\in\{0,1\}^n$, such that $I_\emptyset=[0,1]$, and  $I_{i,0}$ and $I_{i,1}$ share a left, resp.\ right endpoint with $I_i$.

(a) Since each $I_i$ is compact, so is $S_n$, and therefore $S$.

(b) We have $S\subset S_n=\bigcup I_i,$ where \begin{equation}\label{h}\sum_{i\in\{0,1\}^n} h(\diam I_i)=2^n\cdot (1/2^n)=1\end{equation} and $\max_i\diam I_i=h^{-1}(1/2^n)\to 0$ as $n\to\infty$, whence $\mu(S)\le 1$.

For the converse inequality, assume that $S\subseteq\bigcup_{J\in\mathcal J} J$, where $\mathcal J$  is countable. We show that $\sum h(\diam J)\ge 1$. We may assume (by taking convex hull and fattening a bit) that the sets $J$ are open intervals with endpoints not in $S$. By compactness, we may assume that there are only finitely many of them. Now  we may change our mind and assume (by cutting off superfluous bits) that each $J$ is  the convex hull of two intervals $I_i$,  where $i\in\{0,1\}^n$ with a fixed $n$.

In view of \eqref{h}, it suffices to show that \begin{equation}\label{superadd}h(\diam J)\ge\sum_{i\in\{0,1\}^n, I_i\subseteq J}h(\diam I_i)=\frac1{2^n}\sum_{i\in\{0,1\}^n, I_i\subseteq J}1\end{equation} for any such $J$. We use induction on $n$. For $n=1$, we have $J=I_0$ or $J=I_1$, when \eqref{superadd} holds with equality, or $J=[0,1]$, when it holds with strict inequality. Let $n\ge 2$ and assume that \eqref{superadd} holds for $n-1$ in place of $n$, whenever $J$ is the convex hull of two intervals $I_i$, $i\in\{0,1\}^{n-1}$. Let us prove the same for $n$.

If the rightmost $I_i$ contained in $J$ has an index $i\in\{0,1\}^n$ that ends on 0, then let $i'$ be the same $i$ with the last digit modified to 1. Let $J'$ be the convex hull of $J$ and $I_{i'}$. By concavity of $h$, we have $$h(\diam J')-h(\diam J)\le 1/2^n,$$ except, in the hypercube case, if $\diam J'>1/e$, but then we have $\diam J>1/2$ and $h(\diam J)>1$ trivially.

Thus, it suffices to prove \eqref{superadd} for $J'$ in place of $J$. We may perform a similar trick at the left end of $J$. After all, we may assume that that the leftmost interval $I_i$ in $J$ has $i$ ending on 0 and the rightmost one has $i$ ending on 1. But then $J$ is the convex hull of two intervals $I_i$ with $i\in\{0,1\}^{n-1}$ and we are done by the induction hypothesis.
\end{proof}

We continue to treat the hypercube and the projective plane simultaneously. We omit the subscripts cube and proj.
As in the proof of Lemma~\ref{constr}(b), we use $S$ to construct a graphoning.
For $x,y\in [0,1]$, put $W(x,y)=1$ if $|x-y|\in S$ and $W(x,y)=0$ otherwise. Let $\lambda $ be Lebesgue measure on $\mathcal B=\mathcal B[0,1]$.
The tuple $\mathbf G=([0,1], \mathcal B, \la,\mu, W)$ is  an acyclic 1-regular Euclidean graphoning with gauge function $h$.
Acyclicity
means that for any $k\ge 2$, the modulo 1 sum of $k$ independent $\mu$-random elements of $S$
is a.s.\ not in $S$. Here $\mu$-random means that for each block of binary digits in \eqref{Scube} or \eqref{Sproj}, we choose the common value 0 or 1 with equal probability, and we do this independently for all blocks. There will a.s.\ be a 
block where we choose  1  exactly twice for the common value, but for 
 the following $k$ blocks, we choose 0 all $k^2$ times. 
In the $k$-fold modulo 1 sum this block will not consist of equal digits.

\begin{Pro}\begin{itemize}
    \item[(a)]
$(Q_d,d)\to \mathbf G_{\mathrm{cube}}$ as a Euclidean limit as $d\to\infty$.
\item[(b)] If $G_n$ is the incidence graph of a projective plane of order $q_n$, and $q_n\to\infty$, then $$(G_n,q_n+1)\to \mathbf G_{\mathrm{proj}}$$ as a  Euclidean  limit as $n\to\infty$.
\end{itemize}
\end{Pro}

\begin{proof}
From Propositions~~\ref{alphareggraphoning} and \ref{Hausdorff}, 
we have $t(F,\mathbf G)=1$ for any tree $F$.
Since $\mathbf G$ is acyclic, we have  $t(F,\mathbf G)=0$ if $F$ contains  a cycle.
By Propositions~\ref{cubeproj} and \ref{girthdensity}, the convergence claims in the Proposition hold. By  \eqref{growth}, the limiting graphonings are Hausdorff, and therefore, Euclidean  limits.
\end{proof}


\section{Sub-Markov spaces} In the previous subsection, we  dealt with  regular sequences of large essential girth. In this section, 
we shall  
construct limit objects for arbitrary sequences of large essential girth. Sadly, these limit objects will not be graphonings, they will be weaker structures: probability spaces with a reversible sub-Markov kernel --- in other words,  pseudo-graphonings with $W=0$.
\subsection{Tree densities and kernel preserving maps}
 In this subsection, we do the easy part of the preparations. 



We only need to care about tree densities. Recall again that the rooted tree densities, as in  Definition~\ref{rootedhomomdef},   depend only on the sub-Markov kernel $\deg$, not  on the function $W$ --- cf. Proposition~\ref{rootedwelldef}(c) --- or on the probability measure $\la$.  They  satisfy a  simple recursion whose proof is trivial from Definition~\ref{rootedhomomdef}:
\begin{Lemma}\label{rootedtree} Let $\deg$ be a sub-Markov kernel on $(X,\mathcal B)$, and let $x_0\in X$. Let $(F,o)$ be  a rooted tree. We have
$$t((F,o), (X, x_0))=\prod_{i=1}^k \int_X t((F_i,o_i),(X, x_i))\mathrm d\deg (x_0)(x_i)$$ if $F-o$ is the disjoint union of trees $F_1$, \dots, $F_k$ whose nodes adjacent to $o$ in $F$ are $o_1$, \dots, $o_k$ respectively.
\end{Lemma}





\begin{Def} Let $(X,\mathcal A, \deg)$ and $(Y,\mathcal B, \deg)$ be spaces with sub-Markov kernels (by abuse of notation, both kernels are denoted by deg). A measurable map $\phi:X\to\ Y$ is \it kernel preserving \rm if $\phi_* (\deg (x))=\deg (\phi (x))$ for all $x\in X$.
\end{Def}

\begin{Pro}\label{tpreserve}
If $\phi:X\to Y$ is kernel preserving, then  $$t((F,o), (X,x_0))=t((F,o), (Y, \phi(x_0)))$$ for all rooted trees $(F,o)$ and all $x_0\in X$.
\end{Pro}
\begin{proof}
We use induction on $\vv (F)$. Assume that the Proposition is true for all rooted trees with less than $\vv (F)$ nodes. Let $F-o$ be
the disjoint union of trees $F_1$, \dots, $F_k$ whose nodes adjacent to $o$ in $F$ are $o_1$, \dots, $o_k$ respectively. We have $$
t((F,o), (Y, \phi(x_0)))=\prod_{i=1}^k\int_Yt((F_i,o_i),(Y,y_i))\mathrm d\deg (\phi(x_0))(y_i)$$ by Lemma~\ref{rootedtree}. Here we may replace $\deg (\phi(x_0))$ by $\phi_* (\deg(x_0))$ because $\phi$ is kernel preserving.  But $$\int_Yt((F_i,o_i), (Y, y_i))\mathrm d (\phi_*(\deg(x_0)))(y_i)=\int_Xt((F_i,o_i), (Y, \phi(x_i)))\mathrm d \deg(x_0)(x_i)$$ for all $i$ by the definition of $\phi_*$.  By the induction hypothesis, we may replace $t((F_i,o_i), (Y, \phi(x_i)))$ by $t((F_i,o_i), (X, x_i))$. The Proposition follows by using  Lemma~\ref{rootedtree} again.
\end{proof}

The unrooted tree densities of a probability space with a reversible sub-Markov kernel are defined by formula~\eqref{unrootedformula}.
By Proposition~\ref{unrootedwelldef}, 
they are well defined. 
  That proposition is about pseudo-graphonings, but we can always put $W=0$ to get a pseudo-graphoning.

Simultaneously kernel preserving and measure preserving maps also preserve reversibility and unrooted tree densities: 
\begin{Pro}
If $\mathbf G=
(X, \mathcal A, \kappa, \deg)$ and 
 $\mathbf H=
 (Y,\mathcal B, \la, \deg)$ are probability spaces with sub-Markov kernels on each,   $\phi:X\to Y$ is measurable, kernel-preserving and measure-preserving, and the kernel on $X$ is reversible,  then \begin{itemize}
 \item[(a)] the kernel on $Y$ is reversible, and
 \item[(b)] we have $t(F,\mathbf G)=t(F,\mathbf H)$ for all trees $F$.
 \end{itemize}
 \end{Pro}

\begin{proof} (a) For all $A, B\in\mathcal B$, we have $$\int_A\deg_B\mathrm d\la=\int_{\phi^{-1}(A)} (\deg_B\circ \phi)\mathrm d\kappa=
\int_{\phi^{-1}(A)} \deg_{\phi^{-1}(B)}\mathrm d\kappa,$$
which is symmetric w.r.t.\ $A$ and $B$ because the kernel on $X$ is reversible.

\medskip

(b) We have $$t(F,\mathbf G)=\int_Xt((F,o), (\mathbf G, x))\mathrm d\kappa(x)=\int_Yt((F,o), (\mathbf H, y))\mathrm d\la(y)=t(F,\mathbf H)$$ for all rooted trees $(F,o)$.
 \end{proof}

\subsection{The space of consistent measure sequences}
We now wish to construct a compact metrizable space that, for sequences of large essential girth, will play a  role analogous to that of the space of bounded-degree rooted graphs in the Benjamini--Schramm limit theory \cite[Subsection 18.3]{L}.

Given  a compact metric space $K$, let $\mathcal M(K)$ be the space of Borel measures on $K$ whose total mass is $\le 1$ (i.e., sub-probability measures). This, endowed with the L\'evy--Prokhorov metric, is again a  compact metric space, where convergence is the weak convergence of measures.
A continuous map $f:K\to L$ of compact metric spaces induces  a continuous map $f_*:\mathcal M(K)\to\mathcal M (L)$.

Let $
M_0$ be a point and let $
 M_{r}=\mathcal
  M(
  M_{r-1})$. E.g., $
   M_1\simeq [0,1]$. Let $f_0:
    M_1\to
     M_0$ be the unique map, and let $f_r=(f_{r-1})_*:
      M_{r+1}\to
       M_r$.
A \it consistent sequence \rm is a sequence $$\sigma=(\sigma_r)_{r=0}^\infty\in \prod _{r=0}^\infty 
 M_r$$ such that $f_r(\sigma_{r+1})=\sigma_r$ for all $r$. Let $
  M$ be the set of consistent sequences. This is the inverse limit of the spaces $M_r$ with respect to the maps $f_r$. It is closed in the above product space, and therefore compact. Let $\mathcal B$ be the $\sigma$-algebra of Borel sets in $
   M$.

There is a canonical sub-Markov kernel on $(M, \mathcal B)$. Let $$\tilde A=\{\sigma\in M: \sigma _r\in A\} $$ whenever  $A\subseteq M_r$ is   Borel. Let $$\mathcal A=\{\tilde A:A \;\textrm {Borel in}\;  M_r, r=0,1,\dots\}.$$ This is an algebra of sets, and it generates $\mathcal B$ as a  $\sigma$-algebra.
Define $\deg_{\tilde A}(\sigma)=\sigma_{r+1}(A)$ whenever $A\subseteq  M_r$ is a  Borel set and $\sigma\in M$.
Then $\deg (\sigma)$ is a finite measure on $\mathcal A$, therefore it extends to a unique measure on $\mathcal B$
by the Hahn--Kolmogorov Theorem \cite[Section IV.4]{D}. This defines $\deg : M\times \mathcal B\to [0,1]$. The class of sets  $A\in \mathcal B$ such that $\deg_A:M\to[0,1]$ is measurable contains $\mathcal A$ and is closed under monotone sequential limits, therefore equals  $\mathcal B$. Thus, $\deg$ is a sub-Markov kernel.

This sub-Markov kernel, when viewed as a map $\deg :M\to\mathcal M (M)$, is a  homeomorphism.
Indeed, its inverse is given by projecting a  measure $\sigma\in \mathcal M(M)$ to each $M_r$ to get a  consistent sequence
of measures $\sigma_{r+1}\in \mathcal M(M_r)=M_{r+1}$ which we complete by the unique element $\sigma_0$ of $M_0$. This two-sided inverse map of $\deg$ is continuous and $\mathcal M(M)$ is compact, so $\deg$ is a homeomorphism.

A useful consequence of this is
\begin{Lemma}\label{cont} If $f:M\to[0,1]$ is continuous, then so is $\deg_f:M\to [0,1]$.
\end{Lemma}

\begin{proof} Let $x_n\to x$ in $M$. Since $\deg$ is continuous, $\deg(x_n)\to \deg (x)$ weakly. Since $f$ is continuous, $$\deg_f(x_n)=\int_Mf\mathrm d\deg (x_n)\to \int_Mf\mathrm d\deg (x)=\deg_f(x).$$
\end{proof}

We now have a sub-Markov kernel on $M$, so rooted tree homomorphism densities of $M$ are defined.

\begin{Lemma}For a  fixed rooted tree $(F,o)$ of radius $\le r$,  the rooted homomorphism density $$t((F,o), (M,\sigma))=t((F,o), (M_r,\sigma_r))$$ depends only on $\sigma_r$, and this dependence is continuous.
\end{Lemma}

\begin{proof}
Induction on $r$, using Lemmas~\ref{rootedtree} and \ref{cont}.\end{proof}
Let $\mathcal T^\bullet$ be the set of rooted trees such that the root has exactly  one neighbor. 
\begin{Pro}\label{tinjective} The map \begin{equation*}\begin{aligned}t:M & \to [0,1]^{\mathcal T^\bullet}\\
\sigma & \mapsto (t((F,o), (M,\sigma)))_{(F,o)\in\mathcal T^\bullet}
\end{aligned}\end{equation*} is a homeomorphism between $M$ and its image $t(M)$.
\end{Pro}

\begin{proof} Since $M$ is compact and $t$ is continuous, it suffices to prove that $t$ is injective.
Let  \begin{equation*}\begin{aligned}t_r:M_r & \to [0,1]^{\mathcal T_{\le r}^\bullet}\\
\sigma & \mapsto (t((F,o), (M,\sigma)))_{(F,o)\in\mathcal T_{\le r}^\bullet},
\end{aligned}\end{equation*} where $\mathcal T_{\le r}^\bullet$ is the set of rooted trees in $\mathcal T^\bullet$ with radius $\le r$. It suffices to prove that $t_r$ is injective for all $r$.
For $r=0$, this holds because $M_0$ is a point. Assume that it holds for $r$. Let us prove it for $r+1$. Suppose that $t_{r+1}(\sigma_{r+1})=t_{r+1}(\sigma'_{r+1})$ for some $\sigma_{r+1}, \sigma_{r+1}'\in M_{r+1}=\mathcal M(M_r)$.
We need to show that $\sigma_{r+1}= \sigma_{r+1}'$.
We have $$(t_r)_*\sigma_{r+1}
\in
\mathcal M
\left([0,1]
^{\mathcal T_{\le r}^\bullet}\right),$$
and similarly for $\sigma_{r+1}'$. Since $t_r$ is injective, it suffices to prove that these two measures on the cube coincide, or, equivalently, their moments coincide. But a  moment of $(t_r)_*\sigma_{r+1}$  is the same thing as  a homomorphism density $t((F,o), (M_{r+1},\sigma_{r+1}))$, where
$(F,o)\in \mathcal T_{\le r+1}^\bullet$. Indeed, if we think of  $F-o$ as a  family of elements of $\mathcal T^\bullet_{\le r}$  that are glued together at their roots (the roots become the node adjacent to $o$ in $F$), and each rooted tree $(T,p)\in\mathcal T^\bullet_{\le r}$ occurs $m(T,p)$ times in this family, then
 $$t((F,o), (M_{r+1},\sigma_{r+1}))=\int_{M_r}\prod_{(T,p)\in \mathcal T^\bullet_{\le r}} t((T,p), (M_r,\sigma_r))^{m(T,p)}\mathrm d\sigma_{r+1}(\sigma_r).$$
\end{proof}

We now show that $M$ is the terminal object in the category of sub-Markov spaces.

\begin{Pro}
Any space 
with a sub-Markov kernel admits a unique kernel preserving map to $M
$.\end{Pro}

\begin{proof}
Uniqueness is immediate from Propositions~\ref{tpreserve} and \ref{tinjective}.

To prove existence, let $(X, \mathcal B, \deg)$ be a space with a sub-Markov kernel.
We construct a kernel preserving map $\sigma : X\to M$. Let $\sigma_0:X\to M_0$ be the unique map. If $\sigma_r:X\to M_r$ is already defined, then put $$\sigma_{r+1}(x)= (\sigma_r)_*(\deg (x))\in \mathcal M(M_r)=M_{r+1}$$ for all $x\in   X$. Let $\sigma(x)=(\sigma_r(x))_{r=0}^\infty$. This is a consistent sequence, i.e., $f_r(\sigma_{r+1}(x))=\sigma_r(x)$ for all $r$. We show this by induction on $r$. It is true for $r=0$ because both sides are elements of the singleton $M_0$. Let us assume it is true for $r-1$.  Then it is true for $r$ because $$f_r(\sigma_{r+1}(x))=(f_{r-1}\circ\sigma_r)_*(\deg (x))=(\sigma_{r-1})_*(\deg (x))=\sigma_r(x).$$

Thus, we have $\sigma: X\to M$. We must prove that the map $\sigma$ is kernel preserving, i.e., $\sigma_*(\deg(x))=\deg(\sigma(x))$ for all $x$. It suffices to show that these two measures on $M$ coincide when pushed down to $M_r$, for all $r$. Using the definition of $\deg(x)$ on the right hand side, this amounts to $(\sigma_r)_*(\deg(x))=\sigma_{r+1}(x)$. This is true by the very  definition of $\sigma_{r+1}(x)$.
\end{proof}

Let us now examine 
probability measures on $(M,\mathcal B)$ that make the canonical kernel $\deg$ reversible, i.e.,  involution-invariant measures. 
These are  analogous to a  basic concept in the Benjamini--Schramm graph limit theory: involution-invariant probability distributions on the space of rooted graphs with a  degree bound.

\begin{Pro} The set of involution-invariant probability measures $\la$ on $(M,\mathcal B)$ 
 is  closed under affine combinations that are nonnegative measures, and is closed  in the weak topology.
\end{Pro}

\begin{proof}
The measure preserving condition is linear in $\la$, hence remains true for affine combinations.

To prove closedness under weak limits, let $\la_n$ satisfy the measure preserving equation for $n=1,2,\dots$, and let $\la_n\to\la$ weakly.
We prove the equality~\eqref{fg} for $\la$. Using Lemma~\ref{cont}, we get the equality 
  for continuous $f$ and $g$. For a fixed continuous $g$, the class  of measurable $f:M\to [0,1]$ for which the equality holds is closed under monotone pointwise limits by the Monotone Convergence Theorem, thus this class contains all measurable $f$. The same argument for fixed measurable $f$ and varying $g$ finishes the proof.
\end{proof}

We are ready for the main result of this section.
\begin{Th}\label{unique}
Let $(G_n,d_n)$ be an admissible sequence such that $t(F,G_n,d_n)$ converges for all trees $F$. Then there is a unique involution-invariant Borel probability measure $\la$ on $M$ such that
$t(F,G_n,d_n)\to t(F,
 (M,\mathcal B,\la, \deg))$ for all trees  $F$ as $n\to\infty$.
\end{Th}

\begin{proof}
To prove uniqueness, observe that if $\la$ and $\la'$ both have the desired property, then $t(F,\mathbf G)=t(F,\mathbf G')$ for all trees $F$, whence the measures $$t_*\la, t_*\la'\in \mathcal M\left( [0,1]^{\mathcal T^\bullet} \right)$$ have the same moments, so they coincide. Thus, $\la=\la'$.

To prove existence, consider the graphoning $\mathbf G_n=\left(V(G_n), \mathcal P(V(G_n)), \la_n, \mu_n, W_n\right)$ corresponding to $(G_n,d_n)$ by Remark~\ref{graphgraphoning}. Push $\la_n$ forward to $M$ using the unique degree preserving map $\mathbf G_n\to M$. Then push it further to $
[0,1]^{\mathcal T^\bullet}$ using $t$. The resulting sequence of probability measures converges weakly because all  moments converge. The weak limit is a probability measure $\la$ concentrated on $t(M)$ which, when pulled back to $M$ using $t^{-1}:t(M)\to M$, has the desired properties.
\end{proof}
There is a corresponding version of the Aldous--Lyons Conjecture:
\begin{?} Is it true that for every involution-invariant Borel probability measure $\la $ on $M$ there exists a convergent sequence $(G_n,d_n)$ of large essential girth such that
$t(F,G_n,d_n)\to t(F,(M,\mathcal B,\la, \deg))$ for all trees  $F$ as $n\to\infty$ ?
\end{?}
In the Benjamini--Schramm case, the  affirmative answer was proved by G.\ Elek~\cite{E2}.

If $(G_n,d_n)$ is a convergent sequence of large essential girth, then the tree densities carry all the information, so the pseudo-graphoning $$\mathbf G=(M,\mathcal B,\la, W=0, \deg),$$ where $\la$ is given by  Theorem~\ref{unique}, is  a limit  for the sequence. This  may be unsatisfactory: we might want  large essential girth to be reflected in the acyclicity of  the kernel $\deg$ rather than  only in the fact that $W=0$ (because an acyclic $\deg$ would give us some hope of finding a different $W$ that would turn $\mathbf G$ into a true  graphoning with unchanged homomorphism densities). This is easy to achieve, as we explain below. The price to pay is that the new probability measure and reversible sub-Markov kernel will not be on the space $M$, and we lose uniqueness.

If $(X,\mathcal A)$ and $(Y,\mathcal B)$ are two measurable spaces with a sub-Markov kernel on each one, then we get a sub-Markov kernel on $(X\times Y, \mathcal A\otimes \mathcal B)$ by defining the measure $\deg (x,y)$ to be the product of the measures $\deg(x)$ and $\deg(y)$. Then $\deg^k(x,y)$ is the
product of the measures $\deg^k(x)$ and $\deg^k(y)$ for all $k$. Thus, if $\deg^k(x)\perp\deg(x)$, then $\deg^k(x,y)\perp\deg(x,y)$. 

A product of reversible kernels given on two probability spaces  is clearly reversible on the product space. The homomorphism density of any tree in the product will be the product of the densities in the factors.

Thus, we can multiply any probability space endowed with a reversible sub-Markov kernel  by either one of the many acyclic 1-regular graphonings constructed in Subsection~\ref{limcubeproj} to get an acyclic space with unchanged  tree densities.  This proves
\begin{Th}\label{acyclicthm} 
 Let $(G_n,d_n)$ be an admissible sequence such that $t(F,G_n,d_n)$ converges for all trees $F$. Then there exists a probability space  $\mathbf G$ endowed with an acyclic reversible sub-Markov kernel such that $t(F,G_n,d_n)\to t(F,\mathbf G)$ for all trees $F$.
\end{Th}

\section{Regularity lemma?}
To conclude, we briefly speculate on one of the questions involved in Problem~\ref{problem}(a):
 does every convergent sequence $(G_n,d_n)$ tend to a  pseudo-graphoning $\mathbf G$ ? 
 If $(G_n,d_n)$ has large essential girth, the answer is in the affirmative by Theorem~\ref{unique}: choose $W=0$.
In general, the proof of an affirmative answer might involve an appropriate version of Szemer\'edi's Regularity Lemma. The very weak version below is unlikely to suffice.
\begin{Pro}\label{Alon}
For any family $\mathcal G$ of admissible pairs $(G,d)$, 
and for any $\epsilon>0$, there exists an $N$ such that for every
$(G,d)\in\mathcal G$ there exists $(G',d')\in \mathcal G$ with $\vv(G')\le N$, $d'\le N$, and $$|t(F,G,d)-t(F,G',d')|<\epsilon$$ for all $F$ with $\vv(F)\le 1/\epsilon$.
\end{Pro}
In the Benjamini--Schramm setting, i.e., when $\mathcal G$ is the family of pairs $(G,d)$ such that  $d$ is a fixed degree bound, this is equivalent to
\cite[Proposition 19.10]{L}, which answered a question of L.\   Lov\'asz. The very simple proof by Noga Alon carries over easily to Proposition~\ref{Alon}.

\section*{Acknowledgements}
Many thanks to Mikl\'os Ab\'ert for believing in this project, and for stimulating discussions. I am grateful to P\'eter Csikv\'ari, Viktor Kiss,   Bal\'azs Szegedy, and G\'abor Tardos for 
 useful comments. 

\end{document}